\documentclass[sts,preprint,reqno,a4paper]{imsart}
\usepackage{setspace}
\usepackage[utf8]{inputenc}
\usepackage{amsthm,amsmath}
\RequirePackage[colorlinks,citecolor=blue,urlcolor=blue,breaklinks=true]{hyperref}

\usepackage{amssymb}
\usepackage{amsfonts,mathrsfs}
\usepackage{url}
\usepackage{pifont,dsfont,booktabs}
\usepackage{breakcites,microtype}

\usepackage{natbib}
\bibpunct{(}{)}{;}{a}{,}{,}
\usepackage{hyperref}
\hypersetup{
  colorlinks = true,
  urlcolor = blue,
  linkcolor = blue,
  citecolor = blue,
}

\usepackage{cleveref}
\usepackage{autonum}

\usepackage{graphicx}
\usepackage{enumitem}
\usepackage{nicefrac}

\usepackage[longend,algoruled,linesnumbered,boxed]{algorithm2e} 
\DontPrintSemicolon
\usepackage{algorithmic}

\startlocaldefs

\crefname{algocf}{alg.}{algs.}
\Crefname{algocf}{Algorithm}{Algorithms}

\def\bs{\boldsymbol}

\def\g{\gamma}

\def\d{\delta}

\def\e{\epsilon}

\def\l{\lambda}

\def\s{\sigma}

\def\O{\Omega}
\def\ie{\textit{i.e., }}

\def\NN{\mathbb N}

\def\RR{\mathbb R}

\def\bf1{\mathbf 1}
\def\bsw{\boldsymbol w}
\def\bxi{\boldsymbol\xi}
\def\bXi{\boldsymbol\Xi}
\def\bmu{\boldsymbol\mu}
\def\bt{\boldsymbol\theta}

\def\btheta{\boldsymbol\theta}
\def\bTheta{\boldsymbol\Theta}
\def\bDelta{\boldsymbol\Delta}
\def\bY{\boldsymbol Y\!}
\def\bX{\boldsymbol X}
\def\bZ{\boldsymbol Z}
\def\bL{\boldsymbol L}

\def\bfA{\mathbf A}
\def\bfM{\mathbf M}
\def\bfQ{\mathbf Q}
\def\bfU{\mathbf U}
\def\bfI{\mathbf I}
\def\bfJ{\mathbf J}
\def\bfX{\mathbf X}
\def\bfY{\mathbf Y}
\def\bfPi{\mathbf\Pi}
\def\bfT{\mathbf\Theta}

\def\bu{\boldsymbol u}
\def\bm{\boldsymbol{m}}

\def\fcar{\mathds{1}}
\def\pinf{{+\infty}}

\def\esp{\mathbf E}

\def\prob{\mathbf P}

\def\calM{\mathcal M}
\def\calI{\mathcal I}
\def\simiid{\overset{\text{iid}}{\sim}}

\def\hat{\widehat}
\def\tilde{\widetilde}

\def\trace{\mathbf{tr}}
\def\sgn{{\rm sgn}}
\theoremstyle{plain}
\newtheorem{theorem}{Theorem}
\newtheorem{lemma}{Lemma}
\newtheorem{proposition}{Proposition}

\newtheorem*{theorem*}{Theorem}
\newtheorem*{lemma*}{Lemma}
\newtheorem*{proposition*}{Proposition}
\newtheorem*{corollary*}{Corollary}
\theoremstyle{remark}
\newtheorem{remark}{Remark}
\newtheorem*{remark*}{Remark}

\newtheorem*{note*}{Note}
\theoremstyle{definition}

\newtheorem*{definition*}{Definition}

\endlocaldefs

\begin{document}

\begin{frontmatter}

\title{\LARGE Minimax estimation of a multidimensional linear functional in sparse Gaussian models and robust estimation of the mean}
\runtitle{Estimation of multidimensional linear functionals}

\begin{aug}
\author{\fnms{Olivier} \snm{Collier,}\ead[label=e3]{}}
\author{\fnms{Arnak S.} \snm{Dalalyan}\ead[label=e2]{arnak.dalalyan@ensae.fr}}

\runauthor{O.\ Collier and A.\ Dalalyan}

\affiliation{Modal'X, Universit\'e Paris-Nanterre and CREST, ENSAE}

\address{Modal'X, UPL, Univ Paris Nanterre, F92000 Nanterre France
~\\ ENSAE-CREST, 5 Avenue Le Chatelier, 91120 Palaiseau, France.}
\end{aug}

\begin{abstract}
We consider two problems of estimation in high-dimensional Gaussian models.
The first problem is that of estimating a linear functional of the means of
$n$ independent  $p$-dimensional Gaussian vectors, under the assumption that
most of these means are equal to zero. We show that, up to a logarithmic
factor, the minimax rate of estimation in squared Euclidean norm is
between $(s^2\wedge n) +sp$ and $(s^2\wedge np)+sp$. The estimator that attains
the upper bound being computationally demanding, we investigate suitable
versions of group thresholding estimators that are efficiently computable even
when the dimension and the sample size are very large. An interesting new phenomenon
revealed by this investigation is that the group thresholding leads to a
substantial improvement in the rate as compared to the element-wise thresholding.
Thus, the rate of the group
thresholding is $s^2\sqrt{p}+sp$, while the element-wise thresholding has an
error of order $s^2p+sp$. To the best of our knowledge, this is the first known setting in which leveraging the group structure leads to a polynomial
improvement in the rate.

The second problem studied in this work is the estimation of the common
$p$-dimensional mean of the inliers among $n$ independent Gaussian vectors.
We show that there is a strong analogy between this problem and the first one.
Exploiting it, we propose new strategies of robust estimation that are computationally tractable and have better rates of convergence than the other
computationally tractable robust (with respect to the presence of the
outliers in the data) estimators studied in the literature. However, this tractability comes with a loss of the minimax-rate-optimality in some regimes.

\end{abstract}



\end{frontmatter}

\section{Introduction}

Linear functionals are of central interest in statistics. The problems of
estimating a function at given points, predicting the value of a future
observation, testing the validity of a hypothesis, finding a dimension
reduction  subspace are all examples of statistical inference on linear
functionals. The primary goal of this paper is to investigate the problem
of estimation of a particular form of linear functional defined as the sum
of the observed multidimensional signals. Although this problem is of
independent interest on its own, one of our motivations for studying it
is its tight relation with the problem of robust estimation.

Various aspects of the problem of estimation of a linear functional of an
unknown high-dimensional or even infinite-dimensional parameter were studied
in the literature, mostly focusing on the case of a functional taking real
values (as opposed to the vector valued functional considered in the present work).
Early results  for smooth functionals were obtained by \cite{Koshevnik}.
Minimax estimation of linear functionals over various classes and models
were thoroughly analyzed by \cite{donoho1987minimax, klemela2001,
Efromovich1994,Gol04,cai2004,cai2005, laurent2008,butucea2009,juditsky2009}.
There is also a vast literature on studying the problem of estimating quadratic
functionals \citep{DONOHO1990290,laurent2000,cai2006,BickelRitov}. Since the
estimators of (quadratic) functionals can be often used as test statistics, the
problem of estimating functionals has close relations with the problem of testing
that were successfully exploited in \citep{comminges2012, comminges2013,CD11c,
Lepski1999}. The problem of estimation of nonsmooth functionals was also tackled
in the literature, see \citep{cai2011}.

Some statistical problems related to  functionals of high-dimensional parameters 
under various types of sparsity constraints were recently addressed in several papers. The case of real valued
linear and quadratic functionals was studied by \cite{collier2017} and \cite{ColComTV}, focusing on the Gaussian sequence model. \cite{Gassiat} analyzed the problem
of the signal-to-noise ratio estimation in the linear regression model under
various assumptions on the design. In a companion paper of the present submission,
\cite{CollierDal3} considered the problem of a vector valued linear functional
estimation when the observations are drawn from a Poisson distribution. It turns
out that the result established in the present work for the group (hard and soft)
thresholding estimators are valid for the Poisson model as well, but it is not the
case for the results on the greedy estimator studied in \Cref{subsection_GSS}.

We first investigate the order of magnitude of the worst-case risk of 
three types of estimators of a linear functional: the greedy subset selection (GSS), 
the group (hard and soft) thresholding (GHT and GST) and the component-wise
thresholding (HT). We then establish a non-asymptotic lower bound on the minimax 
risk that shows its dependence on the three main parameters of the model: the
sample size $n$, the dimension $p$ and the (column-)sparsity $s$. This
lower bound implies that the greedy subset selection is minimax rate optimal
in the sparse regime $s= O(p\vee\sqrt{n})$, whereas the group thresholding
is minimax rate optimal in the super-sparse case $s= O(\sqrt{p})$. The advantage of the
group thresholding as compared to the greedy subset selection is that the former
is computationally efficient, whereas the latter is not. In all these 
considerations, we neglect logarithmic factors. \Cref{tab:1} summarizes 
our main contributions related to the problem of linear functional estimation.
\begin{table}[ht]
\begin{tabular}{c|c|c|cc}
\toprule
Estimator  & Risk Bound  & 
	Computationally  & Stated in &\\
  & (up to log factors) & 
	 efficient &  &\\	
\midrule
GSS & $sp + s^2\wedge np$ & No & \Cref{theorem_GSS} &\\
GHT & $sp + s^2\sqrt{p}\wedge np$ & Yes & \Cref{theorem_HT} &\\
GST & $sp + s^2\sqrt{p}\wedge np$ & Yes & \Cref{theorem_ST} &\\
HT & $sp + s^2p\wedge np$ & Yes & \Cref{theorem_BT} &\\
\midrule
Lower bound & $sp + s^2\wedge n$ & $-$ &\Cref{theorem_lowerbound} &\\
\bottomrule
\end{tabular}
\caption{A summary of our results related to the estimation of a 
linear functional. The risk is normalized by $\sigma^2$, the variance
of the noise, and the bounds of the second column hide logarithmic 
factors and multiplicative universal constants.}
\label{tab:1}
\end{table}

In particular, one can observe that the ratio of the worst-case risk of the 
group thresholding procedure and that of the component-wise thresholding
might be as small as  $O(p^{-1/2})$. To the best of our knowledge, this is the 
first known setting in which leveraging the group structure leads to
such an important improvement of the rate. In previous results, the improvement was
of at most logarithmic order. Another interesting remark is
that the group soft thresholding estimator we investigate here has a
data-dependent threshold\footnote{Although we do not have a formal proof of that, 
but all the computations we did make us believe that it is impossible to get such a small risk bound for the group
soft thresholding estimator based on a threshold that does not depend on
data.}. Finally, note that while the thresholding 
estimators are natural candidates for solving the problem under consideration
in the sparsity setting, the greedy subset selection is a new procedure
introduced in this paper to get the best known upper bound on the minimax
risk.

A second problem studied in this work is the robust estimation of the mean of
a Gaussian vector. As explained in forthcoming sections, this problem has close
relations to that of estimation of a linear functional. In order to explain this
relation, let us recall that one of the most popular mathematical framework for
analyzing robust estimators is the Huber contamination model \citep{huber1964}. 
It assumes that
there is a reference distribution $\prob_{\bmu}$, parameterized by $\bmu\in\mathcal M$, 
the precise value of which is unknown, and a contamination distribution $\bfQ$, 
which is completely unknown. The data points $Y_i$, $i=1,\ldots,n$ are independent random variables drawn from the mixture distribution $\mathbf P_{\epsilon,\bmu,\bfQ} = 
(1-\epsilon)\prob_{\bmu}+\epsilon\,\bfQ$, where $\epsilon\in[0,1]$ is the rate
of contamination. The goal is then to estimate the parameter $\bmu$, see the 
papers \citep{Gao2015,chen2016} for some recent results. This means that among
the $n$ observations, there are $s$ inliers drawn from $\prob_{\bmu}$  and
$(n-s)$ outliers drawn from $\bfQ$, all these observations being independent and
$s$ being a binomial random variable with parameters $n$ and $(1-\epsilon)$. 
Thus, the specificity of the model is that all the outliers are assumed to be 
drawn from the same distribution, $\bfQ$. 

We suggest here to consider an alternative model for the outliers. In the general
setting, it corresponds to considering the number of outliers, $s$, as a 
deterministic value and to assuming that the outliers $\{Y_i:i\in O\}$ (where
$O\subset [n]$ is of cardinality $s$) are independent and satisfy $Y_i\sim 
\prob_{\bmu_i}$. Thus, we do not assume in this model that the outliers are all generated by the same random mechanism. This model and the Huber model are two
different frameworks for assessing the quality of the estimators. It is quite
likely that in real world applications none of these two models are true. However, 
both of them are of interest for comparing various outlier-robust estimators and
investigating optimality properties.

To explain the connection between the robust estimation and the problem of estimation
of a linear functional, let us consider the contamination model of the previous paragraph. That is, we assume  that the observations $\bY_i$ are independent and drawn
from $\prob_{\bmu_i}$, with $\bmu_i=\bmu$ for every inlier $i\in O^c = 
\{1,\ldots,n\}\setminus O$. In addition, let $\bmu$ be the mean of $\prob_{\bmu}$ and the family $\{\prob_{\bmu}\} $ be translation invariant (meaning that for every 
vector $\bs a$, the random variable $\bY_i-\bs a$ is drawn from $\prob_{\bmu_i-\bs a}$). If we have an 
initial estimator $\hat\bmu_0$ of $\bmu$, which is consistent but not necessarily 
rate-optimal, then we can define the centered observations $\bY'_i = \bY_i-\hat\bmu_0$. 
Each observation $\bY'_i$ will have a distribution close to $\prob_{\btheta_i}$, where 
$\{\btheta_i\triangleq \bmu_i-\bmu,i\in[n]\}$ is a sparse set of vectors, 
so that $\frac1n\sum_{i\in[n]} \bY_i$ is a natural estimator of $\bmu+\frac1n\sum_{i\in[n]} \btheta_i$. The strategy we propose here
is to use an estimator $\hat\bL_n$---based on the transformed observations $\bY_i'$--- of the linear functional $\bL_n = 
\frac1n\sum_{i\in[n]} \btheta_i$ and then to update the estimator of $\bmu$ 
by the formula
$\hat\bmu_1 = \frac1n\sum_{i\in[n]} \bY_i - \hat\bL_n$. This procedure can be
iterated using $ \hat\bmu_1$ as an initial estimator of $\bmu$. We elaborate on
this approach in the case of the normal distribution, $\prob_{\bmu} = 
\mathcal N_p(\bmu,\sigma^2\bfI_p)$, in the second part of the present work. 

\subsection{Organization}
The rest of the paper is organized as follows. \Cref{section_linear} is
devoted to the problem of linear functional estimation. It contains the
statements of the main results concerning the risk bounds of different
relevant estimators and some lower bounds on the minimax risk. The problem
of robust estimation is addressed in \Cref{section_robust}. We summarize our
findings and describe some directions of future research in \Cref{section_conc}.
The proofs of main theorems are postponed to \Cref{section_proofs1}, whereas
the proofs of technical lemmas are gathered in \Cref{section_proofs2}.
Some well-known results frequently used in the present work are recalled
in~\Cref{section_tails}.

\subsection{Notation}
We denote by $[k]$ the set of integers $\{1,\ldots,k\}$. The $k$-dimensional
vectors containing only ones and only zeros are denoted by $\mathbf1_k$ and
$\mathbf 0_k$, respectively. As usual, $\|\bu\|_2$ stands for the Euclidean norm
of a vector $\bu\in\RR^k$. The $k\times k$ identity matrix is denoted by $\bfI_k$.
For every $p\times n$ matrix $\bfM$ and every $T\subset [n]$, we denote by $\bfM_T$
the submatrix of $\bfM$ obtained by removing the columns with indices lying outside
$T$. The Frobenius norm of $\bfM$, denoted by $\|\bfM\|_F$, is defined by
$\|\bfM\|_F^2 = \trace(\bfM^\top\bfM)$. We will use the notation $\bL(\bfM)$ for
the linear functional $\bfM\mathbf 1_n$ equal to the sum of the columns of $\bfM$.

\section{Estimation of a linear functional}\label{section_linear}

We assume that we are given a $p\times n$ matrix $\bfY$ generated by the following model:
\begin{equation}\label{model1}
	\bfY = \bfT+\sigma\bXi, \quad \xi_{i,j} \simiid \mathcal{N}(0,1).
\end{equation}
This means that the deterministic matrix $\bfT$ is observed in Gaussian white noise of variance $\sigma^2$.
Equivalently, the columns $\bY_i$ of $\bfY$ satisfy
\begin{equation}
	\bY_i = \bt_i+\sigma\bxi_i, \quad \bxi_i \simiid \mathcal{N}(\mathbf 0_p,\bfI_p),\quad i=1,\ldots,n.
\end{equation}
Our goal is to estimate the vector $\bL(\bfT)\in\RR^p$, where $\bL:\RR^{p\times n}\to\RR^p$ is the
linear transformation defined by
\begin{equation}\label{L}
	\bL(\bfT) = \sum_{i=1}^n \bt_{i} = \bfT\bf1_n.
\end{equation}
Let us first explain that this is a nontrivial statistical problem, at least when both
$p$ and $n$ are large. In fact, the naive solution to the aforementioned problem consists
in replacing in \eqref{L} the unknown matrix $\bfT$ by the noisy observation $\bfY$.
This leads to the estimator $\hat\bL = \bfY\bf1_n$, the risk of which can be easily
shown to be
\begin{equation}\label{risk1}
	\esp_\bfT \|\hat\bL-\bL(\bfT)\|_2^2 = \sigma^2 np.
\end{equation}
When the matrix $\bfT$ has at most $s$ nonzero columns with $s$ being much smaller than $n$, it is possible to design estimators that perform much better than the naive estimator $\hat\bL_n$. Indeed, an oracle who knows the sparsity pattern $S = \{i\in[n] : \bt_i\not=0\}$ may use the oracle-estimator $\hat\bL_S = \bL(\bfY_S)$ which has a risk equal to $\sigma^2s p$. It is not difficult to show that there is no estimator having a smaller risk uniformly over all the matrices $\bfT$ with a given sparsity
pattern $S$ of cardinality $s$. Thus, we have two benchmarks: the very slow rate $\sigma^2np$ attained by the naive estimator and the fast rate $\sigma^2 sp$ attained by the oracle-estimator that is unavailable in practice. The general question that we study in this work is the following: what is the best possible rate in the range $[\sigma^2 sp, \sigma^2 np]$ that can be obtained by an estimator that does not rely on the knowledge of $S$?

In what follows, we denote by $\calM(p,n,s)$ the set of all $p\times n$ matrices with real entries having at most $s$ nonzero columns:
\begin{equation}\label{calM}
\calM(p,n,s) = \Big\{\bfT\in \RR^{p\times n} : \sum_{i=1}^n \mathds 1(\|\bt_i\|_2>0) \le s\Big\}.
\end{equation}

\subsection{Greedy subset selection}\label{subsection_GSS}

Let us consider a greedy estimator that tries to successively recover various pieces of the sparsity pattern $S$. We start by setting $I_1 = [n]$ and $\calI_1 = \big\{J\subseteq I_1 : \|\bL(\bfY_{J})\|_2^2 \ge 12\sigma^2(|J|p + \l|J|^2)\big\}$. If $\calI_1$ is empty, then we set $\hat J_1 = \varnothing$ and terminate. Otherwise, \ie when $\calI_1$ is not empty, we set $\hat J_1 = \text{arg}\min\big\{|J| : J\in\calI_1\big\}$ and $I_2 = I_1\setminus \hat J_1$. In the next step, we define $\calI_2$, $\hat J_2$ and $I_3$ in the same way using as starting point $I_2$ instead of $I_1$. We repeat this procedure until we get $\hat J_\ell = \varnothing$ or $I_{\ell+1} = \varnothing$. Then we set
\begin{equation}\label{est:greedy}
\hat S = \hat J_1\cup \ldots\cup \hat J_\ell
\quad\text{and}\quad \hat\bL{}^{\rm GSS} = \bL(\bfY_{\hat S}).
\end{equation}
The detailed pseudo-code for this algorithm is given in \Cref{algo:1} below.

\begin{algorithm}[tbh]
\setstretch{1.2}
\SetAlgoLined
\SetKwInOut{Input}{input}\SetKwInOut{Output}{output}
\Input{matrix $\bfY$ and noise variance $\s$, threshold $\l$.}
\Output{vector $\hat\bL{}^{\rm GSS}$.}
initialization $I \gets [n]$ and $\hat S \gets \varnothing$.\;
\Repeat{$I$ is empty or $\hat J$ is empty}{
Set $\calI \gets \big\{J\subseteq I : \|\bL(\bfY_{J})\|_2^2
\ge 12\sigma^2 (|J|p + \l |J|^2)\big\}$.\;
\eIf{$\calI =\varnothing$}{$\hat J\gets\varnothing$\;
}{
Set $\hat J \gets \text{arg}\min\big\{|J| : J\in\calI\big\}$.\;
}
Update $\hat S \gets \hat S \cup \hat J$.\;
Update $I \gets I\setminus \hat J$.\;
}
\Return{$\hat\bL{}^{\rm GSS}\gets \sum_{i\in\hat S} \bY_i$}
\caption{Greedy subset selection algorithm}
\label{algo:1}
\end{algorithm}

\begin{theorem}\label{theorem_GSS}
\label{GSS}
Let $\d\in(0,1)$ be a prescribed tolerance level. The greedy subset selection estimator
with $\l = \nicefrac32\log(\nicefrac{2n}{\d})$ satisfies
\begin{equation}\label{GSS:1}
\sup_{\bfT\in \calM(p,n,s)}\prob_\bfT\Big(\|\hat\bL{}^{\rm GSS}
-\bL(\bfT)\|_2^2
    \le 60\sigma^2{s(p+\l s)}\Big) \ge 1-\d.
\end{equation}
\end{theorem}

This result tells us that the worst-case rate of convergence of the GSS estimator over the class $\calM(p,n,s)$ is $\sigma^2 s(p+s\log n)$. As a consequence, the minimax risk of estimating the functional $\bL(\bTheta)$
over the aforementioned class is at most of order $\sigma^2 s(p+s\log n)$.
As we will see below, this rate is optimal up to a logarithmic factor.

However, from  a practical point of view, the GSS algorithm has limited
applicability because of its high computational cost. It is therefore
appealing to look for other estimators that can be computed efficiently
even though their estimation error does not decay at the optimal rate
for every possible configuration on $(p,n,s)$. Let us note here that
using standard tools it is possible to establish an upper bound similar
to \eqref{GSS:1} that holds in expectation.

\subsection{Group hard thresholding estimator}

A natural approach to the problem of estimating $\bL(\bTheta)$ consits
in filtering out all the signals $\bY_i$ that have a large norm and by
computing the sum of the remaining signals. This is equivalent to solving
the following optimization problem
\begin{equation}\label{definition_HT}
	\hat{\bfT}{}^{\rm GHT} =
	\arg\min_{\mathbf T} \Big\{ \|\bfY-{\mathbf T}\|_F^2 + \lambda^2 \sum_{i=1}^n
	\fcar_{\boldsymbol t_i\neq\mathbf 0} \Big\},
\end{equation}
where $\l>0$ is a tuning parameter. The estimator $\hat{\bfT}{}^{\rm GHT}$,
hereafter referred to as group hard thresholding, minimizes the negative log-likelihood penalized by the number of non-zero columns in $\bTheta$.
One easily checks that the foregoing optimization problem can be solved
explicitly and the resulting estimator is
\begin{equation}
	\hat{\bt}{}^{\rm GHT}_i = \bY_i \fcar_{\|\bY_i\|_2 \ge \lambda},\qquad
	i\in[n].
\end{equation}
Using the group hard thresholding estimator of $\bTheta$ and the method of
substitution, we can estimate $\bL(\bfT)$ by
\begin{equation}\label{definition_LHT}
	\hat{\bL}{}^{\rm GHT}=\bL(\hat{\bfT}{}^{\rm GHT}).
\end{equation}
It is clear that this estimator is computationally far more attractive than
the GSS estimator presented above. Indeed, the computation of the GHT
estimator requires at most $O(pn)$ operations. However, as stated in the
next theorem, this gain is achieved at the expense of a higher
statistical error.

\begin{theorem}\label{theorem_HT}
	Let $\hat{\bL}{}^{\rm GHT}$ be the estimator defined in~\eqref{definition_LHT} with the tuning parameter
	\begin{equation}
		\lambda^2/\sigma^2 = p+4\big\{\log(1+n/s^2)\vee p^{1/2}\log^{1/2}(1+n^2p/s^4)\big\}.
	\end{equation}
	There exists a universal constant $c>0$ such that, for every
	$\bfT\in\calM(p,n,s)$, it holds
	\begin{equation}
		\esp_\bfT\big[\big\| \hat{\bL}{}^{\rm GHT}-\bL(\bfT) \big\|_2^2\big] \le 
		c\sigma^2\Big(s^2p^{1/2}\log^{1/2}(1+n^2p/s^4)+s^2\log(1+n/s^2)+sp\Big).
	\end{equation}
\end{theorem}
Using the fact that $\log(1+x)\le x$, we infer from this theorem that the rate of the 
group hard thresholding for fixed $\s$ is of order $s^2\sqrt{p}\wedge np+sp$, up to 
a logarithmic factor. Moreover, the rate obtained in this theorem can not be improved, 
up to logarithmic factors, as stated in the next theorem.
\begin{proposition}\label{theorem_HT_lowerbound}
Let us denote by $\hat{\bL}{}^{\rm GHT}_\l$ the estimator defined in~\eqref{definition_LHT} with a threshold $\l>0$. There are two universal constants $p_0\in\NN$ and $c>0$, such that for any $p\ge p_0$ and $s\le n/61$, the following lower bound holds
	\begin{equation}
		\inf_{\l>0}\sup_{\bfT\in\calM(p,n,s)} \esp_\bfT \big\| \hat{\bL}{}^{\rm GHT}_\l-\bL(\bfT) \big\|_2^2 \ge c\sigma^2\Big((s^2p^{1/2})\wedge (np)+s^2\log(1+n/s^2)+sp\Big).
	\end{equation}
\end{proposition}

The proofs of these theorems being deferred to \Cref{section_proofs1}, let us
comment on the stated results. At first sight the presence of the sparsity $s$
in the definition of the threshold $\l$ in \Cref{theorem_HT} might seem
problematic, since this quantity is unknown in most practical situations.
However, one can easily modify the claim of \Cref{theorem_HT} replacing $n/s^2$
and $np^{1/2}/s^2$ respectively by $n$ and $np^{1/2}$ both in the definition of
$\l$ and the subsequent risk bound.

A second remark concerns the rate optimality. If we neglect the logarithmic factors in this discussion, the rate of the GHT estimator is shown to be at
most of order $\sigma^2(s^2\sqrt{p}\wedge np +sp)$. This coincides with the optimal rate
(and the one of the GSS estimator) when $s = O(\sqrt{p})$ and has an extra
factor $p$ in the worst-case $p = O(s^4/n^2)$. When there is a limit on the
computational budget, that is when the attention is restricted to the
estimators computable in polynomial (in $s,p,n$) time, we do not know whether
such a deterioration of the risk can be avoided.

An inspection of the proof of \Cref{theorem_HT} shows that if all the nonzero
signals $\bt_i$ are large enough, that is when $\min_{i\in S}\|\bt_i\|_2^2\ge cp $
for some constant $c>0$, the extra factor $\sqrt{p}$ disappears and the GHT
achieves the optimal rate. Put differently, the signals at which the GHT estimator
fails to achieve the optimal rate are those having an Euclidean norm of order $p^{1/4}$. This is closely related to the minimax rate of separation in hypotheses
testing. It is known that the separation rate for testing $H_0:\btheta=\mathbf 0$
against $H_1:\|\bt\|_2\ge \rho$, when one observes $\bY\sim \mathcal N(\bt, \sigma^2\bfI_p)$ is of order $\sigma p^{1/4}$.

Our last remark on \Cref{theorem_HT} concerns the relation with element-wise
hard thresholding. The idea is the following: any column-sparse matrix
$\bTheta$ is also sparse in the most common sense of sparsity. That is,
the number of nonzero entries of the matrix $\bTheta$ is only a small fraction
of the total number of entries. Therefore, one can estimate the entries
of $\bTheta$ by thresholding those of $\bfY$ and then estimate $\bL(\bTheta)$
by the method of substitution. The statistical complexity of this estimator
is quantified in the next theorem, the proof of which is similar to the
corresponding theorem in \citep{collier2017}.
\begin{theorem}\label{theorem_BT}
	Let $\hat{\bL}{}^{\rm HT}$ be the element-wise hard thresholding estimator
	defined by $\hat{\bL}{}^{\rm HT}_i = \sum_{j=1}^n \bY_{i,j} \fcar_{\bY_{i,j}>\lambda}$
	for $i\in [p]$. If the threshold $\l$  is chosen so that $\lambda^2=2\sigma^2\log(1+n/s^2)$, then
	\begin{equation}
		\sup_{\bfT\in\calM(p,n,s)}\esp_\bfT\big[\|\hat{\bL}{}^{\rm HT}-\bL(\bfT)\|^2 \big]
		\le c\sigma^2s^2p\log(1+n/s^2),
	\end{equation}
	where $c>0$ is a universal constant.
\end{theorem}

A striking feature of the problem of linear functional estimation
uncovered by \Cref{theorem_HT} and \Cref{theorem_BT}, is that exploiting the
group structure leads to an improvement of the risk which may attain a factor
$p^{-1/2}$ (for the squared Euclidean norm). To the best of our knowledge, this
is the first framework in which the grouping is proved to have such a strong
impact. This can be compared to the problem of estimating the matrix $\bTheta$
itself under the same sparsity assumptions. Provable guarantees in such a setting
show only a logarithmic improvement due to the use of the sparsity structure
\citep{lounici2011,Bunea2014}.

\subsection{Group-soft-thresholding estimator}

A natural question is whether the results obtained above for the group hard
thresholding can be carried over a suitable version of the soft-thresholding
estimator. Such an extension could have two potential benefits. First, the soft
thresholding is defined as a solution to a convex optimization problem, whereas
hard thresholding minimizes a nonconvex cost function. This difference makes
the soft thresholding method more suitable to deal with various statistical
problems. The simplest example is the problem of linear regression: the extension
of the soft thresholding estimator to the case of non-orthogonal design is the lasso,
that can be computed even when the dimension is very large. In the same problem,
the extension of the hard thresholding is the BIC-type estimator, the computation
of which is known to be prohibitively complex when the dimension is large.

A second reason motivating our interest in the soft thresholding is its smooth
dependence on the data. This smoothness implies that the estimator is less sensitive to the changes in the data than the hard thresholding. Furthermore, it makes it
possible to design a SURE-type algorithm for defining an unbiased estimator of the
risk and, eventually, selecting the tuning parameter in a data-driven way.

In the model under consideration, the group soft thresholding estimator 
$\hat{\bfT}^{\rm GST}$
can be defined as the minimizer of the group-lasso cost function, that is
\begin{equation}\label{definition_ST}
	\hat{\bfT}^{\rm GST} = \arg\min_{\mathbf{T}}
	\Big\{ \sum_{i=1}^n \| \bY_i - \boldsymbol{t}_i \|_2^2 + \sum_{i=1}^n \l_i \|\boldsymbol{t}_i\|_2 \Big\}.
\end{equation}
This problem has an explicit solution given by
\begin{equation}\label{definition_STexplicite}
	\hat{\bt}{}^{\,\rm GST}_i = \Big(1-\frac{\l_i}{2\|\bY_i\|_2}\Big)_+ \bY_i.
\end{equation}
It is natural then to define the plug-in estimator as
$\hat{\bL}{}^{\rm GST} = \bL(\hat{\bfT}^{\rm GST})$. The next theorem establishes
the performance of this estimator.

\begin{theorem}\label{theorem_ST}
	The estimator $\hat{\bL}{}^{\rm GST}=\bL(\hat{\bfT}^{\rm GST})$ defined
	in~\eqref{definition_STexplicite} with\footnote{Note that $\lambda_i=+\infty$
	if $\|\bY_i\|^2_2\le\sigma^2p$. This reflects
	the fact that there is no need to fit the signals of very low magnitude.}
	\begin{equation}
	\l_i = \frac{2\sigma\gamma\|\bY_i\|_2}{\big(\|\bY_i\|^2_2-\sigma^2p\big)_+^{1/2}}, \quad
	\gamma^2 = 
	4\big\{\log(1+n/s^2)\vee p^{1/2}\log^{1/2}(1+n^2p/s^4)
	\big\}
	\end{equation}
	satisfies, for every $\bfT\in\calM(p,n,s)$,
	\begin{equation}
	\esp_\bfT \big[\big\|\hat{\bL}{}^{\rm GST}-\bL(\bfT)\big\|^2 \big]
	\le c\sigma^2\Big(s^2p^{1/2}\log^{1/2}(1+n^2p/s^4)+s^2\log(1+n/s^2)+sp\Big),
	\end{equation}
	where $c>0$ is some universal constant.
\end{theorem}

The comments made after the statement of \Cref{theorem_HT} can be repeated here.
The dependence of $\g$ on $s$ is not crucial; one can replace $s$ by 1 in the expression
for $\g$, this will not have a strong impact on the risk bound. The bound in expectation can be
complemented by a bound in deviation. The rate obtained for the soft thresholding is
exactly of the same order as the obtained in \Cref{theorem_HT} for the group 
hard thresholding. A notable difference, however, is that in the case of soft thresholding the tuning
parameter $\l$ suggested by the theoretical developments is data dependent.


\subsection{Lower bounds and minimax rate optimality}\label{section_mmx}

We now address the question of the optimality of our estimators. In~\citep{collier2017}, the case $p=1$ was solved with lower and upper bounds matching  up to a constant. In particular, Theorem~1 in \citep{collier2017} yields the following proposition.
\begin{proposition}
	Assume that $s\in[n]$, then there is a universal constant $c>0$ such that
	\begin{equation}
		\inf_{\hat{\bL}} \sup_{\bfT\in\calM(1,n,s)} \esp_\bfT \big( \hat{\bL}-\bL(\bfT) \big)^2 \ge c \sigma^2 s^2 \log(1+n/s^2).
	\end{equation}
\end{proposition}
Note that when $n=s$, this rate is of the order of $\sigma^2s$. It is straightforward that this rate generalizes to $\sigma^2sp$ in the
multidimensional case. Furthermore, if we knew in advance the sparsity pattern $S$, then we could restrict the matrix of observations to the indices in $S$, and we would get the oracle rate $\sigma^2sp$. These remarks are made formal in the following theorem.
\begin{theorem}\label{theorem_lowerbound}
	Assume that $1\le s\le n$, then there is a universal constant $c>0$ such that
	\begin{equation}
		\inf_{\hat{\bL}} \sup_{\bfT\in\calM(p,n,s)} \esp_\bfT \big\| \hat{\bL}-\bL(\bfT) \big\|^2 \ge c \Big[ \sigma^2 s^2 \log(1+n/s^2)+\sigma^2 sp \Big].
	\end{equation}
\end{theorem}
Therefore, the greedy subset selector in \Cref{subsection_GSS} is
provably rate-optimal in the case $s=O(\sqrt{n})$. A question that
remains open is the rate optimality when $\sqrt{n}=O(s)$. The lower bound
of \Cref{theorem_lowerbound} is then of order $\sigma^2(n+sp)$, whereas the
upper bound of \Cref{theorem_GSS} is of order $\sigma^2(s^2+sp)$. Taking into
account the fact that the naive estimator $\bL(\bfY)$ has a risk of order
$\sigma^2 np$, we get that the minimax risk is upper bounded by $\sigma^2(s^2\wedge
np+sp)$.Thus, there is a gap of order $p$ when $p+\sqrt{n} = O(s)$.

\begin{center}
\begin{minipage}{0.95\textwidth}
\centering
\begin{algorithm}[H]
\setstretch{1.2}
\SetAlgoLined
\SetKwInOut{Input}{input}\SetKwInOut{Output}{output}
\Input{matrix $\bfY$, noise variance $\sigma^2$ and confidence level $\d$.}
\Output{vector $\hat\bL{\!\!\!}^{\rm adGSS}$.}
Set ${\rm dist} \gets \sigma^{-1}\|\hat\bL{}^{\rm GSS}-\bL(\bfY)\|_2$.\;
Set $\l \gets \nicefrac32 \log(\nicefrac{4n}{\d})$.\;
Set
$\hat s \gets \min\Big\{k\in[n]: {\rm dist}\le \sqrt{60k(p+\l k)}+\sqrt{n(2p+3\log(2/\delta))}\Big\}$
(if the set is empty, set $\hat s \gets n$)\;
\eIf{$60\hat s(p+\l \hat s)\le 2np+3n\log(2/\delta)$}
{$\hat\bL{}^{\rm adGSS}\gets \hat\bL{}^{\rm GSS}$\;}
{$\hat\bL{}^{\rm adGSS}\gets \bL(\bfY)$\;}
\Return{$\hat\bL{}^{\rm adGSS}$}
\caption{Adaptive GSS}
\label{algo:2}
\end{algorithm}
\end{minipage}
\end{center}

Note that none of the estimators discussed earlier in this work attain the
upper bound $\sigma^2(s^2\wedge np+sp)$; indeed, the latter is obtained as the
minimum of the risk of two estimators. Interestingly, one can design a single
estimator that attains this rate. Previous sections contain all the necessary
ingredients for this. We will illustrate the trick in the case of the GSS
estimator, but similar technique can be applied to any estimator for which an
``in deviation'' risk bound is established.

The idea is to combine the GSS estimator and the naive estimator
$\hat\bL=\bL(\bfY)$, with the aim of choosing the ``best'' one.
The combination can be performed using the Lepski method \citep{Lepski91},
also known as intersection of confidence intervals \citep{Golden}.
The method is described in \Cref{algo:2}. The construction is
based on the following two facts:
\begin{enumerate}
    \item The true value $\bL(\bTheta)$ lies with probability $1-\d/2$ in the
    ball $\mathcal B(\bL(\bfY);r_1)$ with $(r_1/\s)^2={2np+3n\log(2/\delta)}$.
    \item The true value $\bL(\bTheta)$ lies with probability $1-\d/2$ in the
    ball $\mathcal B(\hat\bL{}^{\rm GSS};r_2)$ with $(r_2/\s)^2=60s(p+\l s)$ (cf.\ \Cref{theorem_GSS}).
\end{enumerate}
These two facts imply that with probability at least $1-\d$ the balls
$\mathcal B(\bL(\bfY);r_1)$ and $\mathcal B(\hat\bL{}^{\rm GSS};r_2)$
have nonempty intersection. As a consequence, in this event, we have
$\|\bL(\bfY)-\hat\bL{}^{\rm GSS}\|_2 \le r_1+r_2$ and, therefore, $\hat s\le s$.
Now, if $60\hat s(p+\l \hat s)\le 2np+3n\log(2/\delta)$, then
$\hat\bL{}^{\rm adGSS}=\hat\bL{}^{\rm GSS}$ and we have
\begin{align}
    \|\hat\bL{}^{\rm adGSS}-\bL(\bTheta)\|_2 =  \|\hat\bL{}^{\rm GSS}-\bL(\bTheta)\|_2 \le r_2
\end{align}
along with
\begin{align}
    \|\hat\bL{}^{\rm adGSS}-\bL(\bTheta)\|_2
        &=  \|\hat\bL{}^{\rm GSS}-\bL(\bfY)\|_2 +\|\bL(\bfY)-\bL(\bTheta)\|_2\\
        &\le \{\sigma\sqrt{60\hat s(p+\l\hat s)}+ r_1\} +r_1\le 3r_1.
\end{align}
Thus, $\|\hat\bL{}^{\rm adGSS}-\bL(\bTheta)\|_2 \le 3 (r_1\wedge r_2)$. In the second case, $60\hat s(p+\l \hat s)\ge 2np+3n\log(2/\delta)$, we have
$\|\hat\bL{}^{\rm adGSS}-\bL(\bTheta)\|_2 =  \|\bL(\bfY)-\bL(\bTheta)\|_2 \le r_1=r_1\wedge r_2$, where the last equality follows from the fact that
$\hat s\le s$. Thus, we have established the following result.

\begin{proposition}
Let $\d\in(0,1)$ be a prescribed confidence level. With probability at least
$1-\d$, the adaptive greedy subset selection estimator
$\hat\bL{}^{\rm adGSS}$ defined in \Cref{algo:2} satisfies
$\|\hat\bL{}^{\rm adGSS}-\bL(\bTheta)\|_2 \le 3\s
\big\{(60sp+90 s^2\log(4n/\delta))\wedge (2np+3n\log(2/\delta))\big\}^{1/2}$.
\end{proposition}

Let us  summarize the content of this section. We have established a
lower bound on the minimax risk, showing that the latter is at least of
order $sp+s^2\wedge n$, up to a logarithmic factor. We have also
obtained upper bounds, which imply that the minimax risk is at most
of order $sp+s^2\wedge (np)$. Furthermore, this rate can be attained by a single estimator (adaptive greedy subset selection).

\section{The problem of robust estimation}\label{section_robust}

The problem of linear functional estimation considered in the previous section
has multiple connections with the problem of robust estimation of a Gaussian mean.
In the latter problem, the observations $\bY_1,\ldots,\bY_n$ in $\RR^p$ are assumed
to satisfy
\begin{equation}\label{definition_observation}
\bY_i = \bmu +\btheta_i+\sigma\bxi_i, \quad \bxi_i\simiid \mathcal{N}(\mathbf 0,\bfI_p),
\end{equation}
where $\bfI_p$ is the identity matrix of dimension $p\times p$.
We are interested in estimating the vector $\bmu$, under the assumption that
most vectors $\btheta_i$ are equal to zero. All the observations $\bY_i$ such that
$i\in S=\{\ell:\|\btheta_\ell\|_2=0\}$ are considered as inliers, while all the others
are outliers. In this problem, the
vectors $\bt_i$ are unknown, but their estimation is not our primary aim. They are
rather considered as nuisance parameters. In some cases, it might be helpful to
use the matrix notation of \eqref{definition_observation}:
\begin{equation}\label{definition_observation1}
\bfY = \bmu\mathbf 1_n^\top +\bTheta+\sigma\bXi.
\end{equation}
The obvious connection with the problem considered in the previous section is
that if we know that $\bmu=\mathbf 0_p$ in \eqref{definition_observation}, then
we recover model \eqref{model1}. This can be expressed in a more formal way as
shown in the next proposition.

\begin{proposition}
The problem of estimating the linear functional $\bL_n(\bTheta) = (\nicefrac1n)\sum_{i\in[n]}
\bt_i$ in model \eqref{definition_observation1} is not easier, in the minimax sense, than that of estimating $\bmu$. More precisely, we have
\begin{align}
\frac{\sigma^2 p}{n}\le \inf_{\hat\bmu}\sup_{\bmu,\bTheta} \esp[\|\hat\bmu-\bmu\|_2^2]\le
2\inf_{\hat\bL_n}\sup_{\bTheta} \esp[\|\hat\bL_n-\bL_n(\bTheta)\|_2^2]
+\frac{2\sigma^2 p}{n},
\end{align}
where the sup in the left-hand side and in the right-hand side are taken,
respectively, over all $\bTheta\in\calM(p,n,s)$ and over all $(\bmu,\bTheta)\in
\RR^p\times\calM(p,n,s)$.
\end{proposition}
\begin{proof}
The first inequality is a consequence of the fact that when all the entries of
$\bTheta$ are zero, the optimal estimator of $\bmu$ in the minimax sense is the
sample mean of $\bY_i$'s. To prove the second inequality, let $\hat\bL_n$ be an
estimator of $\bL_n(\bTheta)$. We can associate with $\hat\bL_n$
the following estimator of $\bmu$: $\hat\bmu(\hat\bL_n) = \bL_n(\bfY)
-\hat\bL_n$. These estimators satisfy
\begin{align}
\esp[\|\hat\bmu(\hat\bL_n)-\bmu\|_2^2]
		&= \esp[\|\bL_n(\bfY)-\hat\bL_n-n\bmu\|_2^2]\\
		&=\esp[\|\bL_n(\bTheta)+\sigma\bL_n(\bXi)-\hat\bL_n\|_2^2]\\
		&\le 2\esp[\|\bL_n(\bTheta)-\hat\bL_n\|_2^2] + {2\sigma^2}\esp[\|\bL_n(\bXi)\|_2^2].
\end{align}
Since $\bL_n(\bXi)$ is drawn from the Gaussian distribution $\mathcal N_p(\mathbf 0
,(\nicefrac1n)\bfI_p)$, we have  $\esp[\|\bL_n(\bXi)\|_2^2] = p\sigma^2/n$ and the claim of the
proposition follows.
\end{proof}

Another important point that we would like to mention here is the relation between
model \eqref{definition_observation} and the Huber contamination model \citep{huber1964}
frequently studied in the statistical literature (we refer the reader to
\cite{Gao2015,chen2016} for recent overviews). Recall that in Huber's contamination model, 
the observations $\bX_1,\ldots,\bX_n$ are $n$ iid $p$-dimensional vectors drawn 
from the mixture distribution $(1-\frac{s}{n}) \mathcal N_p(\bmu,\bfI_p) + 
\frac{s}{n}\,\bfQ$. The particularity of this model is that it assumes all the 
outliers to be generated by the same distribution $\bfQ$; the latter, however, 
can be an arbitrary distribution on $\RR^p$. In contrast with this, our model 
\eqref{definition_observation} allows for a wider heterogeneity of the outliers. 
On the downside, our model assumes that the outliers are blurred by a
Gaussian noise that has the same covariance structure as the noise that corrupts the
inliers. The relation between these two models is formalized in the next result.

\begin{proposition}
Let $\hat\bmu:\RR^{p\times n}\to\RR^p$ be an estimator of $\bmu$ that can be applied
both to the data matrix $\bfX = [\bX_1,\ldots,\bX_n]$ from Huber's model and to $\bfY$ 
from our model \eqref{definition_observation}. Then, we have
\begin{align}
\underbrace{\sup_{\bfQ}\esp_{\bmu,\bfQ}[\|\hat\bmu(\bfX)-\bmu\|_2^2]}
_{\text{\rm risk in the Huber model}}\le
\esp_{\hat s\sim \mathcal B(n,s/n)}
\bigg[\underbrace{\sup_{\bTheta\in\calM(n,p,\hat s)}
\esp_{\bmu,\bTheta}\big[\|\hat\bmu(\bfY)-\bmu\|_2^2\big]
}_{\text{\rm risk in our model \eqref{definition_observation}}}\bigg].
\end{align}
The supremum of the left-hand side is over all probability distributions $\bfQ$ on $\RR^p$
such that\footnote{We denote by $*$ the convolution of the distributions.} 
$\bfQ = \bfQ_0 * \mathcal N_p(\mathbf 0, \sigma^2\bfI_p)$, 
while the notation $\mathcal B(n,s/n)$ stands for the binomial distribution.
\end{proposition}

The proof of this proposition is a simple exercise and is left to the reader.
Although some statistical problems of robust estimation in a framework of the
same spirit as \eqref{definition_observation} have been already tackled in
the literature \citep{DK12,DalalyanC12,BalDal15b,Nguyen, Klopp2017, CherapanamjeriG16}, the entire picture in terms
of matching upper and lower bounds is not yet available. On the other side,
it has been established in \citep{Gao2015} that the minimax rate of estimating
$\bmu$ in Huber's contamination model is
\begin{align}
r^{\rm all}_{\rm mmx}({n,p,s}) = \sigma^2\Big(\frac{p}{n}\vee \frac{s^2}{n^2}\Big).    
\end{align}
It is shown that this rate is achieved by the Tukey median, \textit{i.e.}, 
the minimizer of Tukey's depth. An important observation is that the evaluation of
Tukey's median is a hard computational problem: there exists no algorithm to 
date capable of approximating Tukey's median in a number of operations that scales 
polynomially in $p,n$ and the approximation precision. The best known
computationally tractable robust estimator, the element-wise median, 
has a rate of order \citep[Prop.\ 2.1]{Gao2015}
\begin{align}
\sigma^2\Big(\frac{p}{n}\vee \frac{s^2p}{n^2}\Big).    
\end{align}
We shall show in this section that a suitable adaptation of the
group soft thresholding estimator presented in the previous section
leads to a rate that can be arbitrarily close to
\begin{align}
\sigma^2\Big(\frac{p}{n}\vee\frac{s^2}{n^2}\vee \frac{s^4p}{n^4}\Big).
\end{align}
This shows that if we restrict our attention to the estimators that
have a computational complexity that is at most polynomial, the minimax
rate satisfies, for every $\nu\in(0,1/\log p)$, 
\begin{align}
\sigma^2\Big(\frac{p}{n}\vee \frac{s^2}{n^2}\Big)\lesssim 
r^{\rm poly}_{\rm mmx}({n,p,s})\lesssim 
\sigma^2\Big(\frac{p}{n}\vee\frac{s^2}{n^2}\vee \Big\{\frac{s^4p}{n^4}\Big\}^{1-\nu}\Big),
\end{align}
where $\lesssim$ means inequality up to logarithmic factors.

\subsection{Maximum of profile likelihood with group lasso penalty}

A computationally tractable estimator that allows to efficiently deal 
with structured sparsity and has provably good statistical complexity
is the group lasso \citep{Yuan_Lin_2006,Lin_2006,Hebiri_2008,Meier_2009,lounici2011}. 
We define the group-lasso estimator by
\begin{align}\label{definition_hatmu}
(\hat\bmu,\hat\bfT) \in \arg\min_{\bm,\mathbf{T}}
\Big\{ \sum_{i=1}^n \| \bY_i -\bm - \boldsymbol{t}_i \|_2^2 + \sum_{i=1}^n \l_i \|\boldsymbol{t}_i\|_2 \Big\}.
\end{align}
where the $\l_i$ are some positive numbers to be defined later. The estimator $\hat\bmu$ can
be seen as the maximum of a profile penalized likelihood, where the penalty is proportional 
to the $\ell_{2,1}$ norm (also known as the group lasso penalty) of the nuisance parameter 
$\bTheta$. The above optimization problem is convex and can be solved numerically even when 
the dimension and the sample size are large. It is also well known that $\hat\bmu$ from 
\eqref{definition_hatmu} is exactly the Huber M-estimator \citep[Section 6]{Donoho2016}. 
In addition, these estimators can
also be written as
\begin{align}
	\hat{\bmu} &= \frac1n \sum_{i=1}^{n} \big( \bY_i - \hat{\bt}_i \big)
	=\bL_n(\bfY)-\bL_n(\hat\bTheta), \\
	\hat\bfT &\in \arg\min_{\mathbf{T}} \Big\{ \sum_{j=1}^p \|\bfPi(\bY^j-\boldsymbol{t}^j) \|_2^2 + \sum_{i=1}^n \l_i \|\boldsymbol{t}_i\|_2 \Big\}, \label{definition_groupestimator}
\end{align}
where $\bfPi$ denotes the orthogonal  projection in $\RR^n$ 
onto the orthogonal complement of the constant vector $\mathbf 1_n$. 
Unfortunately, we were unable to establish a risk bound for this estimator
that improves on the element-wise median. The best result that we get is the following.
\begin{theorem}\label{theoreme_bornesuperieure1}
	Consider the estimators of $\bTheta$ and $\bmu$ defined in~\eqref{definition_hatmu} with $\lambda^2 = 32\sigma^2 p + 256\sigma^2 \log(n/\delta)$.
	Then, with probability at least $1-\d$ and provided that 
	$s\le n/32$, we have
	\begin{align}\label{theoreme_bornesuperieure1_1}
	\|\bTheta-\hat\bTheta\|_F^2 &\le 9 s\lambda^2,\quad 
	\|\bL_n(\hat\bTheta)-\bL_n(\bTheta)\|_2^2 \le \frac{288 s^2\lambda^2}{n^2}\\
	\|\hat\bmu-\bmu\|_2^2 &\le \frac{288 s^2\lambda^2}{n^2}+\frac{4\sigma^2 p}{n}
	+\frac{8\sigma^2\log(2/\delta)}{n}.
	\end{align}
\end{theorem}
This result, proved in \Cref{proofs_robust}, shows that the rate of the profiled 
penalized likelihood estimator of $\bmu$, with a group lasso penalty, converges at the
rate $\sigma^2\big(\frac{s^2p}{n^2}\vee \frac{p}{n}\big)$, which coincides with the one
obtained\footnote{To be precise, \citep{Gao2015} establish only a lower bound for
the element-wise median, but a matching upper bound can be proved as well.} by \citep{Gao2015}.
In the rest of this section, we will propose an estimator which improves on this rate. 
To this end, we start with obtaining a simplified expression for the group lasso
estimator $\hat\bTheta$.

First, using the fact that 
$\bfPi\boldsymbol{t}^j = \boldsymbol{t}^j - (\bfI_n-\bfPi)\boldsymbol{t}^j$, we get 
$\|\bfPi(\bY^j-\boldsymbol{t}^j) \|_2^2 = \|\bfPi\bY^j-\boldsymbol{t}^j \|_2^2 -
(\nicefrac1n) \big(\mathbf 1_n^\top\boldsymbol{t}^j\big)^2$, so that
\begin{equation}
	\hat{\bfT} \in \arg\min \Big\{ \sum_{i=1}^n \| (\bfY\bfPi)_i-\boldsymbol{t}_i \|_2^2 - \frac1n \Big\| \sum_{i=1}^n \boldsymbol{t}_i \Big\|_2^2 + \sum_{i=1}^n \l_i \|\boldsymbol{t}_i\|_2 \Big\}.
\end{equation}
Recall that $\bL_n(\bTheta) = (\nicefrac1n) \bL(\bTheta)$. The first-order necessary 
conditions imply that, for every $i$ such that $\hat{\bt}_i\neq\mathbf 0_p$,
\begin{equation}
	-2\big( (\bfY\bfPi)_i-\hat{\bt}_i \big) - 2 \bL_n(\hat{\bTheta}) + \frac{\l_i \hat{\bt}_i}{\|\hat{\bt}_i\|_2} =\mathbf  0_p.
\end{equation}
Furthermore, $\hat\bt_i=\mathbf 0_p$ if and only if
$\big\|2 (\bfY\bfPi)_i+ 2\bL_n(\hat{\bTheta}) \big\|_2\le \l_i$. 
We infer that
\begin{equation}
	\hat{\bt}_i = \frac{(\bfY\bfPi)_i+\bL_n(\hat{\bTheta}) }{\big\|(\bfY\bfPi)_i+
	\bL_n(\hat{\bTheta}) \big\|_2}~\Big( \big\|(\bfY\bfPi)_i+\bL_n(\hat\bTheta) \big\|_2 - \frac{\l_i}2 \Big)_+
\end{equation}
for every $i$. Finally, denoting $\bZ_i=(\bfY\bfPi)_i+\bL_n(\hat{\bTheta})$,  we get
\begin{equation}\label{equation_grouplasso}
	\hat{\bt}_i = \bZ_i \Big( 1-\frac{\l_i}{2\|\bZ_i \|_2} \Big)_+.
\end{equation}
This formula shows the clear analogy between the group lasso estimator $\hat\bTheta$ 
and the soft thresholding estimator studied in the previous section. This analogy
suggests to choose the tuning parameters in a data driven way; namely, it is tempting
to set 
\begin{align} \label{Z:0}
\l_i = \frac{2\gamma\sigma\|\bZ_i\|_2}{(\|\bZ_i\|_2^2-\sigma^2p)_+^{1/2}}
\qquad\Longrightarrow
\qquad
\hat{\bt}_i = \bZ_i \Bigg( 1-\frac{\gamma\sigma}{(\|\bZ_i\|_2^2-\sigma^2p)_+^{1/2}} \Bigg)_+.
\end{align}
Unfortunately, such a choice is impossible to realize since this $\l_i$ depends on the 
solution $\hat\bTheta$ of the optimization problem, which in turn is defined through $\l_i$. 
To circumvent this problem, we suggest to use an iterative algorithm that starts from
an initial estimator $\hat\bL_n$ of $\bL_n(\bTheta)$, defines the vectors 
$\bZ_i = (\bfY\bfPi)_i+\hat\bL_n$ and then updates $\hat\bL_n$ by the formula
$\hat\bL_n = \bL_n(\hat\bTheta)$, where the columns of the matrix $\hat\bTheta$ are 
defined by the second equality in \eqref{Z:0}. This algorithm, called iterative
soft thresholding, is described in \Cref{algo:3}. 
\begin{algorithm}[hbt]
\setstretch{1.2}
\SetAlgoLined
\SetKwInOut{Input}{input}\SetKwInOut{Output}{output}\SetKwInOut{Parameter}{parameters}
\Input{matrix $\bfY$, noise variance $\s$, number of outliers $s$.}
\Parameter{number of iterations $N$, confidence level $\d$.}
\Output{vectors $\hat\bL_n{\!\!\!}^{\rm IST}$ and $\hat\bmu^{\rm IST}$.}
initialization\; 
\qquad $\hat\bTheta \gets $ solution of \eqref{definition_groupestimator} with 
$\lambda^2 = 32\sigma^2(p+8\log(n/\delta))$\;
\qquad $\epsilon\gets  \sqrt{288}\,\displaystyle\frac{s\lambda}{n\sigma}$\;
\For{$k=1,\ldots,N$}{
Set $\hat\bL_n \gets \bL_n(\hat\bTheta)$\;
Set $\gamma^2\gets 8\epsilon^2+4\sqrt{4\epsilon^4+p\epsilon^2}$\;
\For{$i=1,\ldots,n$}{
Set $\bZ_i \gets (\bfY\bfPi)_i+\hat\bL_n$\;
Set 
$\displaystyle
\btheta_i \gets \bZ_i\bigg(1-\frac{\sigma\gamma}{(\|\bZ_i\|_2^2-
\frac{n-1}{n}\sigma^2 p)^{1/2}_+}\bigg)_+.
$
}
{
update\;
\qquad $\hat\bTheta \gets [\hat\btheta_1,\ldots,\hat\btheta_n]$.\;
\qquad $\epsilon\gets \nicefrac{4}{n}(\,s\gamma+ s + \sqrt{sp} + 
\{2s\log(4/\delta)\}^{1/2})$\;
}
}
\Return{$\hat\bL{}^{\rm IST}_n\gets \bL_n(\hat\bTheta)$ and 
$\hat\bmu{}^{\rm IST}=\bL_n(\bfY)-\hat\bL{}^{\rm IST}_n$.} 
\caption{Iterative Soft Thresholding}
\label{algo:3}
\end{algorithm}

Prior to stating the theorem that describes the statistical complexity of this estimator,
we present a result that explains why such an iterative algorithm succeeds in improving
the convergence rate.

\begin{proposition}\label{theoreme_bornesuperieure2}
	Let us set $\bZ_i = (\bfY\bfPi)_i+\tilde\bL_n$, where $\tilde\bL_n$ is a
	preliminary estimator of $\bL_n(\bTheta)$. Let $\d\in(0,1)$ be a tolerance level.
	Consider the estimator of $\bfT$ defined by concatenating the vectors
	\begin{align}\label{Z0}
	\hat\bt_i = \bZ_i \bigg(1-\frac{\sigma\gamma}{(\|\bZ_i\|_2^2-\frac{n-1}{n}
	\sigma^2 p)_+^{1/2}}\bigg)_+,
	\end{align}
	where $\gamma^2 > 4\log(4n/\delta)+4 \{p\log(4n/\delta)\}^{1/2}$ is a tuning parameter.
	Define the event
	\begin{align}
	\Omega_1 &= \Bigg\{\|\tilde\bL_n-\bL_n(\bTheta)\|_2	< 
	\frac{\sigma\gamma^2}{4\sqrt{p+\gamma^2}}\Bigg\}.
	\end{align}
	There is an event $\bar\O$ (the same for all estimators $\tilde\bL_n$) of probability 
	at least $1-\d$, such that on $\O_1\cap\bar\O$, we have
	\begin{align}
		\|\bL(\hat{\bTheta}-\bTheta) \|_2
		\le 4\sigma\Big(\,s\gamma+ s + \sqrt{sp} +  \{2s\log(4/\delta)\}^{1/2}\Big).
			\label{Z6}
	\end{align}
\end{proposition}
It follows from this theorem that at each iteration of the algorithm 
we improve the precision of estimation of $\bL_n(\bTheta)$. Indeed, if $\epsilon_k$ 
is an upper bound on the error $\|\hat\bL_n\!{}^{(k)}-\bL_n(\bTheta)\|_2/\s$ at the $k$th 
iteration, then we get from the last theorem that
\begin{align}\label{eq:recursion}
\epsilon_{k+1} \le \frac{8s}{n}
\big(2\epsilon_k^2+(4\epsilon_k^4+p\epsilon_k^2)^{1/2}\big)^{1/2} +a,
\end{align}
with $a= (4/n)(s+\sqrt{sp}+\{2s\log(4/\delta)\}^{1/2})$.

\begin{lemma}\label{lemma:recursion}
If $\epsilon_0^2\le p$, $n\ge 33s$ and $a\le 0.5\sqrt{p}$, then
\begin{align}\label{recursion:claim}
\epsilon_k \le  \Big\{p^{1/2}\Big(\frac{33^2 s^2}{n^2}\Big)^{1-(\nicefrac12)^{k}}\Big\}
\vee 2a.
\end{align}
\end{lemma}

Combining all these results, we arrive at the following risk bound for the
iterative soft thresholding estimator.

\begin{theorem}
Let $\d\in(0,1)$, $N\in\NN$ and let $\hat\bL{}^{\rm IST}_n(N)$ 
be the iterative soft thresholding estimator obtaind after $N$ iterations. 
Assume that $p\ge \log(8/\delta)$ and $N\ge \log\log p$. There are some universal 
strictly positive constants $c_1,c_2,c_3$ such that if the condition
\begin{align}
s\le c_1 n
\end{align}
is satisfied then, with probability at least $1-2\d$, the following inequalities hold
true:
\begin{align}
\|\hat\bL{}^{\rm IST}_n(N)-\bL_n(\bTheta)\|_2
	&\le c_2\sigma\bigg\{p^{1/2}\Big(\frac{s^2}{n^2}\Big)^{1-2^{-N}}+
	\frac{s+\sqrt{sp}}{n}\bigg\},\\
\|\hat\bmu{}^{\rm IST}-\bmu\|_2
	&\le c_3\sigma\bigg\{p^{1/2}\Big(\frac{s^2}{n^2}\Big)^{1-2^{-N}}+
	\frac{s}{n}+\Big(\frac{p}{n}\Big)^{1/2}\bigg\}.
\end{align}
\end{theorem}
This implies, in particular, that if $p\le C (n/s)^{2-\nu}$
for some $\nu\in(0,1/2)$ close to zero, then performing $N = \log_2(1/\nu)$ 
iterations of the IST algorithm we will recover the mean $\bmu$ of the inliers 
at an optimal rate $(s/n)^2\vee (p/n)$.

To complete this section, let us briefly note that one can use the Lepski method
as described in \Cref{section_mmx} for getting an estimator of $\bmu$ that does
not require the knowledge of $s$. This will only increase the error by a factor
at worst equal to $3$.

\begin{remark}
From an intuitive point of view, the algorithm described in \Cref{algo:3} can be
seen as an iterative approximation of the estimator
\begin{align}\label{Hub}
\hat\bmu^*\in\text{arg}\min_{\bmu\in\RR^p} \sum_{i=1}^n \rho_{\rm H}
\Bigg(\frac{(\|\bY_i-\bmu\|_2^2-\sigma^2p)^{1/2}_+}{\sigma\gamma}\Bigg),
\end{align}  
for an appropriately chosen tuning parameter $\g>0$, where $\rho_{\rm H}$
is the Huber function. Unfortunately, the cost function in the above  
minimization problem is not convex with respect to the parameter $\bmu$. 
This implies that general purpose guarantees available for approximating
solutions of convex programs are not applicable to \eqref{Hub}. To the best
of our knowledge, there is no efficient algorithm that provably approximates
$\hat\bmu^*$. 
\end{remark}

\section{Conclusion and perspectives}\label{section_conc}

In this work, we have studied two problems: the problem of estimating
a multidimensional linear functional and the one of estimating the mean 
of $p$-variate random vectors when the data is corrupted by outliers. 
In the first problem, we have obtained upper and lower bounds on the minimax
risk that match in most situations. More importantly, in both problems, we 
have studied computationally tractable estimators and have obtained the
best known rates of convergence. A surprising outcome of our work is that
exploiting the group structure of the sparsity is far more important in
the problem of linear functional estimation rather than in the problem
of the whole signal. We have also designed a new robust estimator of the
mean that iteratively performs group soft thresholding on a suitable
transformation of the data. 

There several questions related to the present work that remain open. 
First, it would be interesting to close the gap in the minimax rate of
estimation of a linear functional when $p+\sqrt{n}=O(s)$. Second, in both
problems studied in this work, a challenging question for future research
is to establish lower bounds on the minimax risk over computationally
tractable estimators. For the problem of robust estimation, one may use 
a suitable version of the median of means 
\citep{Oliveira,minsker2015,devroye2016,LL17}. We are not aware of any 
result establishing upper bounds on the risk of these methods in the 
models considered in the present work yielding a better rate than 
those presented herein.

%
 %
%

\section{Proofs of the main results}\label{section_proofs1}

This section contains the proofs of the main theorems stated in the previous sections.

\subsection{Proofs of the theorems of \Cref{section_linear} }
\begin{proof}[Proof of \Cref{theorem_GSS}]
Using the triangle inequality several times, we get
\begin{align}
\|\hat\bL{}^{\rm GSS}-\bL(\bfT)\|_2
    &\le \|\bL(\bfY_{\hat S})-\bL(\bfY_{S})\|_2+ \|\bL(\bfY_{S})- \bL(\bfT_S)\|_2\\
    &\le \|\bL(\bfY_{\hat S\setminus S})\|_2+\|\bL(\bfY_{S\setminus \hat S})\|_2
			+ \sigma\|\bL(\bxi_{S})\|_2 \phantom{\Big(}\\
    &\le \sigma\|\bL(\bxi_{\hat S\setminus S})\|_2+\|\bL(\bfY_{S\setminus \hat S})\|_2
			+ \sigma\|\bL(\bxi_{S})\|_2.
    \label{line:3}
\end{align}
To upper bound the three terms of the right hand side, we introduce the event
\begin{equation}
\Omega_\l = \Big\{\|\bL(\bxi_J) \|_2^2\le 2|J|(p+\l |J|)\text{ for all }
J\subseteq [n]\Big\}.
\end{equation}
We will show that the following three claims are true for the tuning parameter $\l$ chosen as in the statement of the theorem.
\begin{description}\itemsep=5pt
\item[Claim 1:] On the event $\Omega_\l$, at least half of the elements of each $\hat J_\ell$
belong to the true sparsity pattern $S$. Thus $|\hat S|\le 2s$.
\item[Claim 2:] $\|\bL(\bfY_{S\setminus\hat S})\|_2^2\le 12\sigma^2 s(p+\l s)$.
\item[Claim 3:] The probability of $\Omega_\l$ is close to $1$.
\end{description}
Let us first show that these claims imply the claim of the theorem. Indeed,  the second term of the right hand side of \eqref{line:3} is bounded by $\sigma\sqrt{12s(p+\l s)}$ in view of Claim 2. The third term is bounded by $\sigma\sqrt{2s(p+\l s)}$ on the event $\O_\l$. Concerning the first term, we know that on $\O_\l$ it is bounded by $\sigma\sqrt{2|\hat S|(p+\l |\hat S|)}$. In view of Claim 1, $|\hat S|\le 2s$. All these inequalities imply that
\begin{align}
\|\hat\bL{}^{\rm GSS}-\bL(\bfT)\|_2^2
    &\le 60\sigma^2{s(p+\l s)}
\end{align}
on the event $\O_\l$. This is exactly the claim of the theorem.

Let us prove now Claims 1-3. To prove the first claim, let us assume that there is a set $J$ among $\hat J_1,\ldots,\hat J_\ell$ and a subset $J_0\subset J$ of cardinality\footnote{To avoid uninteresting and irrelevant technicalities, we assume here that $|J|$ is even.} $|J|/2$ such that $J_0\subset S^c$. This readily implies that $\|\bL(\bfY_{J})\|_2^2\ge 12\sigma^2 |J|(p+\l |J|)$ and $\|\bL(\bfY_{J\setminus J_0})\|_2^2< 6\sigma^2 |J|(p+\l |J|/2)$. Using the additivity of $\bL$ and the triangle inequality, we get
\begin{align}
\|\bL(\bfY_{J})\|_2^2
    & \le (\|\bL(\bfY_{J\setminus J_0})\|_2+ \sigma\|\bL(\bxi_{J_0})\|_2)^2 \phantom{\frac32} \\
    & \le \frac32\|\bL(\bfY_{J\setminus J_0})\|_2^2+ 3\sigma^2\|\bL(\bxi_{J_0})\|_2^2\\
    & < 9\sigma^2 |J|(p+\l |J|/2)+ 3\sigma^2 |J|(p+\l|J|/2) \phantom{\frac32}\\
    & < 12\sigma^2 |J|(p+\l|J|). \phantom{\frac32}
\end{align}
This is in contradiction with the fact that $J$ is one of the sets $\hat J_1,\ldots,\hat J_L$. So, Claim 1 is proved.

The proof of Claim 2 is simpler. By construction, the set $S\setminus \hat S$ is a subset of $I_L$, where $L$ is the number of steps performed by the algorithm. Since the algorithm terminated after the $L$th step, this means that $\calI$ was empty, which implies  that $\|\bL(\bfY_{S\setminus\hat S})\|_2^2\le 12\sigma^2 |S\setminus\hat S|(p+\l|S\setminus\hat S|) \le 12\sigma^2 s(p+\l s)$.

It remains to prove Claim 3. This can be done using the union bound and tail bounds for $\chi^2_p$-distributed random variables. Indeed, we have
\begin{align}
\prob_\bfT\big(\Omega_\lambda^c\big)
    &\le \sum_{k=1}^n \prob_\bfT\big(\exists J\subset [n]\text{ s.t. } |J|=k\text{ and }
            \|\bL(\bxi_J)\|_2^2> 2|J|(p+\l|J|)\big)\\
    &\le \sum_{k=1}^n {{n}\choose{k}} \max_{J:|J|=k}
            \prob_\bfT\big(\|\bL(\bxi_J)\|_2^2> 2k(p+\l k)\big)\\
    &\le \sum_{k=1}^n {{n}\choose{k}}
            \prob\big(\eta > 2p + 2 \l k\big)
\end{align}
where $\eta\sim \chi^2_p$. Using the well known bound on the tails of the $\chi^2_p$
distribution, we get
\begin{align}
\prob\big(\Omega_\lambda^c\big)
    &\le \sum_{k=1}^n {{n}\choose{k}} e^{-2\l k/3} = (1+e^{-2\l/3})^n-1.
\end{align}
Therefore, for $\lambda = \nicefrac32\log(\nicefrac{2n}{\delta})$, we obtain that  $\prob_\bfT\big(\Omega_\lambda^c\big)\le \delta$. This completes the proof of Claim~3 and of the theorem.
\end{proof}

\begin{proof}[Proof of \Cref{theorem_HT}]
	Recall that $\bXi =(\bxi_1, \ldots,\bxi_n)$. First, we decompose
	\begin{align}
		\hat{\bL}{}^{\rm GHT}-\bL(\bfT)
		&= \bL(\bY_S-\bTheta) - \bL(\bY_{S\setminus S_\lambda} )+ \bL(\bY_{S_\l\setminus S})\\
		&= \sigma\bL(\bXi_S) - \bL(\bY_{S\setminus S_\lambda} )+ \sigma\bL(\bXi_{S_\l\setminus S})\\
		&= \sigma\bL(\bXi_{S\cap S_\lambda}) - \bL(\bTheta_{S\setminus S_\lambda} )+ \sigma\bL(\bXi_{S_\l\setminus S}).
		\label{decomposition}
	\end{align}
	so that
	\begin{align}\label{ineq:2}
	\big\| \hat{\bL}{}^{\rm GHT}-\bL(\bfT) \big\|_2 &\le \sigma\|\bL(\bXi_{S\cap S_\lambda})\|_2 +
	\|\bL(\bTheta_{S\setminus S_\lambda})\|_2 + \sigma\|\bL(\bXi_{S_\l \setminus S})\|_2.
	\end{align}
	The first term corresponds to the stochastic error of estimating the signal vectors   
	that are correctly identified as nonzero. We can write
	\begin{equation}
	\|\bL(\bXi_{S\cap S_\lambda})\|_2 = \|\bXi_{S}\mathbf 1_{S\cap S_\lambda}\|_2\le
 	\sqrt{s}\|\bXi_{S}\|.
	\end{equation}
	The second-order moment of the spectral norm of the random matrix $\bXi_S$ can be
	evaluated using well-known upper bounds on the spectral norm of matrices with
	independent Gaussian entries, recalled in \Cref{lemma_Wishart2} below, so that
	\begin{equation}\label{esperance_norm}
	\esp_\bfT \big[\|\bL(\bXi_{S\cap S_\lambda})\|_2^2\big] \le 3s^2 + 3sp+12s.
	\end{equation}
 	Set $\eta_i=\btheta_i^\top\bxi_i/\|\btheta_i\|_2$. We can control the second term in \eqref{ineq:2} using the following inequality
    \begin{align}
	\|\bL(\bTheta_{S\setminus S_\lambda})\|_2&\le \sum_{i\in S} \|\bt_i\|_2 \fcar_{\|\bt_i\|^2<\lambda^2	-2\sigma\bt_i^\top \bxi_i-\sigma^2\|\bxi_i\|^2} \\
	&\le s(\lambda^2-\sigma^2p)^{1/2}+2\sigma\sum_{i\in S} |\eta_i| + \sigma\sum_{i\in S} \big| \|\bxi_i\|^2-p\, \big|^{1/2}.
	\end{align}
	This readily yields
	\begin{equation}\label{eval_on_support}
	\esp_\bfT \big[\|\bL(\bTheta_{S\setminus S_\lambda})\|_2^2]^{1/2}
	\le s(\lambda^2-\sigma^2p)^{1/2}+2\sigma s + \sigma s (2p)^{1/4}.
	\end{equation}
	The third term in \eqref{ineq:2}
	corresponds to the Type II error in the problem of support estimation. Denoting
	$t=(\lambda^2-\sigma^2p)/\sigma^2$, and using tail bounds for the chi-squared
	random variables (see \Cref{lemme_esperancechideux} below), we get
	\begin{align}
	\esp_\bfT[\|\bXi_{S_\l \setminus S}\|_2^2]
	    &= \esp_\bfT\bigg[\bigg\| \sum_{i\in S^c} \bxi_i\fcar(\|\bxi_i\|_2\ge \l/\s)\bigg\|_2^2\bigg]\\
	    &= \sum_{i\in S^c}\esp_\bfT[\|\bxi_i\|_2^2\fcar(\|\bxi_i\|_2^2\ge p+t)]\\
	    &\le 	2n \big( pe^{-t^2/32p}\,\fcar_{t< 4p} + t
	    e^{-t/4}\fcar_{t \ge 4p}\big).
	\end{align}
	Using the fact that $t= 4\log(1+n/s^2)\vee\big\{16p\log(1+n^2p/s^4)\big\}^{1/2}$
	we arrive at
	\begin{align}
	\esp_\bfT[\|\bXi_{S_\l \setminus S}\|_2^2]
	    &\le 	\Big(8s^2\log(1+n/s^2)\Big)\vee \Big(2s^2\sqrt{p}
			\wedge 2np\Big).
	\end{align}
	The result follows from the previous upper bounds and the choice of $\l$.
\end{proof}

\begin{proof}[Proof of \Cref{theorem_HT_lowerbound}]
	We define $\bfT$ as the matrix with entries $\e=\sigma p^{-1/4}$ in the first $s$ 
	columns and $0$ elsewhere. Using the inequality 
	$\forall a,b,~(a-b)^2\ge a^2/2-b^2$ and~\eqref{decomposition}, we get
	\begin{align}
	\esp_\bfT \big[\big\| \hat{\bL}{}^{\rm GHT}-\bL(\bfT) \big\|_2^2 \big]
			&\ge \frac12 \esp_\bfT \big[\big\| \sigma\bL(\bXi_{S_\l\setminus S}) -
			\bL(\bTheta_{S\setminus S_\lambda})\big\|_2^2\big] - \sigma^2\esp_\bfT 
			\big[\big\| \bL(\bXi_{S\cap S_\lambda}) \big\|_2^2\big].
	\end{align}
	Moreover, $\bL(\bXi_{S_\l\setminus S})$ being centered and independent of 
	$\bL(\bTheta_{S\setminus S_\lambda})$, we can develop 
	\begin{align}\label{decomposition_preuve_lowerboundHT}
	\esp_\bfT \big[\big\| \sigma\bL(\bXi_{S_\l\setminus S}) -
			\bL(\bTheta_{S\setminus S_\lambda})\big\|_2^2\big] &= 
			\sigma^2\esp_\bfT \big[\big\|  \bL(\bXi_{S_\l\setminus S})\big\|_2^2\big] +
			\esp_\bfT \big[\big\| \bL(\bTheta_{S\setminus S_\lambda})\big\|_2^2\big].
	\end{align}
	First assume that $\lambda^2\ge \sigma^2 p + \sigma^2\sqrt{p}$ and focus on the second term in 
	the right-hand side of the last display. Using Jensen's inequality, we have
	\begin{align}
	\esp_\bfT \big[\big\| \bL(\bTheta_{S\setminus S_\lambda})\big\|_2^2\big]
	&\ge \big\| \esp_\bfT \big[\bL(\bTheta_{S\setminus S_\lambda})\big]\big\|_2^2\\
	&=	p\e^2s^2 \max_{i\in S}\prob_\bfT\big(\|\bY_i\|_2<\l\big)^2\\
	&= s^2\sqrt{p}\max_{i\in S} \prob_\bfT\big(\|\bY_i\|_2<\l\big)^2.
	\end{align}
  On the other hand, since  $\sgn(\bxi_i^\top\btheta_i)$ 
	is a Rademacher random variable independent of $\|\bxi_i\|_2^2$, for every $i\in S$, 
	we have
	\begin{equation}
	\prob_\bfT\big(\|\bY_i\|_2^2<\lambda^2\big) 
		\ge 0.5\prob\big(\sigma^2\|\bxi_1\|_2^2-\sigma^2p<\sigma^2\sqrt{p}-p\e^2\big) 
		= 0.5\prob\big(\|\bxi_1\|_2^2-p<0\big).
	\end{equation}
	This last probability converges to $1/2$, so that it is larger than 
	$1/4$ for all $p$ large enough.
	
	In the other case, $\lambda^2< \sigma^2 p + \sigma^2\sqrt{p}$, we consider the first term:
	\begin{align}
	\esp_\bfT \big[\big\|  \bL(\bXi_{S_\l\setminus S})\big\|_2^2\big] 
		&= (n-s)p - (n-s)\esp [\|\bxi_1\|_2^2 \fcar_{\|\bxi_1\|_2\le\l/\sigma}] \\
		&\ge (n-s)p - (n-s)(p+\sqrt{p}) \prob\big(\|\bxi_1\|_2^2-p\le\sqrt{p}\big).
	\end{align}
	The probability in the right-hand side converges to $\Phi(2^{-1/2})\le 0.8$, 
	so that for $p\ge 64$ and $s\le n/2$, we have
	\begin{equation}
	\esp_\bfT \big[\big\|  \bL(\bXi_{S_\l\setminus S})\big\|_2^2\big] 
	\ge (n-s) (0.2p-0.8p^{1/2}) \ge 0.05 np.
	\end{equation}
	
	Finally, according to \eqref{esperance_norm}, we have for $p$ large enough
	\begin{align}
		\esp_\bfT \big[\|\bL(\bXi_{S\cap S_\lambda})\|_2^2\big] \le 3s^2 + 3sp+12s
		\le \frac{s^2\sqrt{p}}{65} + 3sp
	\end{align}	
	so that
	\begin{equation}
	\esp_\bfT \big[\big\| \hat{\bL}{}^{\rm GHT}-\bL(\bfT) \big\|_2^2\big] 
	\ge \frac{s^2\sqrt{p}}{65}\wedge(0.05np) - 3sp.
	\end{equation}
	We have to distinguish between two cases. If $s^2\sqrt{p}\le 196sp$, then the result holds in view of~\Cref{theorem_lowerbound}. In the opposite case, the result holds as long as $s\le n/61$.
\end{proof}

\begin{proof}[Proof of \Cref{theorem_ST}] To ease notation, for every random vector
$X$ we write $\|X\|_{L_2}$ for $(\esp[\|X\|_2^2])^{1/2}$. We first notice that
	\begin{align}
	\hat{\bL}{}^{\rm GST}-\bL(\bTheta) &=  \sum_{i\in S} \bt_i \bigg\{\bigg(1-\frac{\sigma\gamma}
	{(\|\bY_i\|^2_2-\sigma^2p)_+^{1/2}}\bigg)_+-1\bigg\} \qquad &(:=T_1)\\
	&+ \sigma\sum_{i\in S} \bxi_i \bigg(1-\frac{\sigma\gamma}{\big(\|\bY_i\|^2_2-\sigma^2p\big)_+^{1/2}}\bigg)_+ \qquad &(:=T_2)\\
	&+ \sigma\sum_{i\not\in S} \bxi_i \bigg(1-\frac{\sigma\gamma}{\sigma\big(\|\bxi_i\|^2_2-p\big)_+^{1/2}}\bigg)_+\qquad &(:=T_3),
	\end{align}	
	so that we only need to bound the expected squared norms of the three terms in the right-hand side. These three terms have the following meanings: the first one is the bias of estimation
or the approximation error, the second term is the stochastic error on the support $S$,
whereas the third term is the stochastic error on $S^\complement$.
	
\paragraph{Evaluation of the approximation error}
For the first term, we use the Minkowski inequality as follows
	\begin{align}\label{eqT1}
	\|T_1\|_{L_2}\le \Big\|\sum_{i\in S} \bt_i \fcar_{\|\bY_i\|_2^2\le\sigma^2 (p+\gamma^2)
	}\Big\|_{L_2} + \sigma\Big\|\sum_{i\in S}  \bt_i \frac{\gamma\,\fcar_{\|\bY_i\|_2^2>\sigma^2 
	(p+\gamma^2)}}{\big(\|\bY_i\|^2_2-\sigma^2p\big)^{1/2}}\Big\|_{L_2}.
	\end{align}
	The first part can be treated exactly as in \eqref{eval_on_support} of the proof 
	of~\Cref{theorem_HT}, \ie
	\begin{equation}\label{eqT1a}
		\Big\|\sum_{i\in S} \bt_i \fcar_{\|\bY_i\|_2^2\le\sigma^2 (p+\gamma^2)
	}\Big\|_{L_2} \le \sigma s \big(\gamma+2 +  (2p)^{1/4}\big).
	\end{equation}
	For assessing the second term in the right hand side of \eqref{eqT1}, 
	we set 
	\begin{align}
	T_{1,i} = \frac{\|\bt_i\|_2\gamma\,\fcar_{\|\bY_i\|_2^2>\sigma^2 (p+\gamma^2)}}
	{(\|\bt_i\|^2_2+2\sigma\bt_i^\top\bxi_i+\sigma^2\big(\|\bxi_i\|_2^2-p\big))^{1/2}} .
	\end{align}
	We consider two cases. The first case corresponds to 
	$\big|2\sigma\bt_i^\top\bxi_i+\sigma^2\big(\|\bxi_i\|_2^2-p\big)\big|<\|\bt_i\|_2^2/2$. 
	In this case one easily checks that  $T_{1,i}\le \sqrt{2}\,\gamma$. In the second case, 
	$\big|2\sigma\bt_i^\top\bxi_i+\sigma^2\big(\|\bxi_i\|_2^2-p\big)\big|\ge\|\bt_i\|_2^2/2$, 
	we have 
	\begin{equation}
		\|\bt_i\|_2^2/4 \le 2\s|\bt_i^\top\bxi_i| \quad \text{or} \quad \|\bt_i\|_2^2/4 \le 
		\sigma^2\big|\|\bxi_i\|_2^2-p\big|.
	\end{equation}
	This readily implies that 
	\begin{equation}
		T_{1,i}= \frac{\|\bt_i\|\gamma\,\fcar_{\|\bY_i\|_2^2>\sigma^2 (p+\gamma^2)}}
	  {\big(\|\bY_i\|^2_2-\sigma^2p\big)^{1/2}}\le \|\bt_i\| \le 8\sigma 
		\frac{|\bt_i^\top\bxi_i|}{\|\bt_i\|_2}+ 2\sigma\sqrt{|\|\bxi_i\|_2^2-p|}.
	\end{equation}
	Therefore, using the fact that $\esp[|\bt_i^\top\bxi_i|^2] = \|\bt_i\|_2^2$ 
	and $\esp[|\|\bxi_i\|_2^2-p|^2] = 2p$, we get	
	\begin{equation}
		\|T_{1,i}\|_{L_2} \le 8\sigma + 2\sigma (2p)^{1/4}.
	\end{equation}
	Combining this inequality with \eqref{eqT1} and \eqref{eqT1a}, we arrive at
	\begin{align}
		\|T_1\|_{L_2} &\le \sigma 
		s\gamma+2\sigma s + \sigma s(2p)^{1/4} + 8\sigma s+ 2\sigma s(2p)^{1/4}\\
		&= \sigma s\gamma+10\sigma s + 3\sigma s(2p)^{1/4}.\label{eqT1c}
	\end{align}
	
	\paragraph{Evaluation of the stochastic error on $S$} 
	The second term can be treated
	exactly as in the proof of Theorem~\ref{theorem_HT} using Wishart matrices, \ie
	\begin{equation}
	\|T_2\|_{L_2}^2 \le \sigma^2 s\esp[\|\bXi_S\|^2] \le \sigma^2(3s^2+3sp+12s).
	\end{equation}
	
	\paragraph{Evaluation of the stochastic error on $S^\complement$} For the third term,
	we write
	\begin{align}
		\|T_3\|_{L_2}^2=\sigma^2\bigg\| \sum_{i\not\in S} \bxi_i
		\bigg(1-\frac{\gamma}{\big(\|\bxi_i\|^2_2-p\big)_+^{1/2}}\bigg)_+\bigg\|_{L_2}^2
			&= \sigma^2\sum_{i\not\in S} \bigg\|\bxi_i\bigg(1-\frac{\gamma}{(\|\bxi_i\|^2_2
				-p)_+^{1/2}}\bigg)_+\bigg\|_{L_2}^2\\
			&\le \sigma^2\sum_{i\not\in S} \esp\big[\|\bxi_i\|_2^2 \fcar_{\|\bxi_i\|_2^2> p+
			\gamma^2}\big].
	\end{align}
  We conclude by~\Cref{lemme_esperancechideux} that the last term satisfies
	\begin{equation}
		\|T_3\|_{L_2}^2\le 
		2n\sigma^2  \big( pe^{-\gamma^4/32p}\,\fcar_{\gamma^2< 4p} +
		\gamma^2e^{-\gamma^2/4}\fcar_{\gamma^2\ge 4p}\big).
	\end{equation}
	Using the fact that $\gamma^2= 4\log(1+n/s^2)\vee\big\{16p\log(1+n^2p/s^4)
	\big\}^{1/2}$ we arrive at
	\begin{align}
	\|T_3\|_{L_2}^2
	    &\le \sigma^2	\Big(8s^2\log(1+n/s^2)\Big)\vee \Big(2s^2\sqrt{p}
			\wedge 2np\Big).
	\end{align}
This completes the proof of the theorem. \end{proof}

\subsection{Proofs of the theorems of \Cref{section_robust}}\label{proofs_robust}

This section gathers the proofs all of the results concerning the problem of robust
estimation of a Gaussian mean. 

\begin{proof}[Proof of \Cref{theoreme_bornesuperieure1}]
	In view of~\eqref{definition_groupestimator}, we have
	\begin{equation}
	\big\| (\bTheta-\hat{\bTheta})\bfPi + \sigma\bXi\bfPi \big\|_F^2 + \l \sum_{i=1}^n \|\hat\bt_i\|_2 \le \sigma^2 \| \bXi\bfPi \|_F^2 + \l \sum_{i=1}^n \|\bt_i\|_2.
	\end{equation}
	Developing the left-hand side, this yields
	\begin{align}
	\| (\bTheta-\hat{\bTheta})\bfPi \|_F^2 &\le 2\sigma\sum_{n=1}^n 
	(\bXi\bfPi)_i^\top(\hat\bt_i-\bt_i) + \l \sum_{i=1}^n \big\{ \|\bt_i\|_2 - 
	\|\hat\bt_i\|_2\big\}.
	\end{align}
	Using the Cauchy-Schwarz inequality, we have on the event
	\begin{equation}\label{definition_evenementA}
	\mathcal{A} = \Big\{ \max_{i\in [n]} \big\| (\bXi\bfPi)_i \big\|_2 \le \frac{\lambda}{4\sigma} \Big\}
	\end{equation}
	that
	\begin{align}
	\|(\bTheta-\hat{\bTheta})\bfPi\|_F^2
	&\le \frac\l2 \sum_{i=1}^n \Big\{ \|\hat\bt_i-\bt_i\|_2
			+ 2 \|\bt_i\|_2 - 2 \|\hat\bt_i\|_2 \Big\} \\
	&\le \frac{3\lambda}2 \sum_{i\in S} \|\hat\bt_i-\bt_i\|_2 -
			\frac\l2 \sum_{i\not\in S} \|\hat\bt_i-\bt_i\|_2,
	\end{align}
	where we used the triangular inequality. The last inequality implies that
	\begin{align}\label{theoreme_bornesuperieure1_2}
	\sum_{i\not\in S} \|\hat\bt_i-\bt_i\|_2\le 3\sum_{i\in S} \|\hat\bt_i-\bt_i\|_2.
	\end{align}
	Furthermore, using the Cauchy-Schwarz inequality, we get
	\begin{equation}
	\|(\bTheta-\hat\bTheta)\bfPi\|_F^2 \le \frac{3\lambda}2 \sqrt{s} \|\bTheta-\hat\bTheta\|_F.
	\end{equation}
	We can then apply Lemma~\ref{lemme_Pi}, since~(\ref{theoreme_bornesuperieure1_2}) 
	ensures that the condition is satisfied with $a=3$: provided that $s\le n/16$, we have
	\begin{equation}
	\frac12 \|\hat\bTheta-\bTheta\|_F^2 \le \frac{3\lambda}2 \sqrt{s} \|\hat\bTheta-\bTheta\|_F,
	\end{equation}
	so that
	\begin{equation}
	\|\hat\bTheta-\bTheta\|_F \le 3 \l \sqrt{s}.
	\end{equation}
	The first claim follows now from~\Cref{lemme_probabiliteA}. To show the second inequality,
	it suffices to remark that
	\begin{align}
	\|\hat\bmu-\bmu\|_2^2
			& \le \frac2{n^2}\|(\hat\bTheta-\bTheta)\mathbf 1_n\|_2^2 +2\sigma^2\|\bL_n(\bXi)\|_2^2\\
			& \le \frac{2}{n^2}\bigg(\sum_{i\in [n]}\|\hat\btheta_i-\btheta_i\|_2\bigg)^2 
				+2\sigma^2\|\bL_n(\bXi)\|_2^2\\
			& \le \frac{2}{n^2}\bigg(4\sum_{i\in S}\|\hat\btheta_i-\btheta_i\|_2\bigg)^2 
				+2\sigma^2\|\bL_n(\bXi)\|_2^2\\
			& \le \frac{32s}{n^2}\sum_{i\in S}\|\hat\btheta_i-\btheta_i\|_2^2 
				+2\sigma^2\|\bL_n(\bXi)\|_2^2\\				
			& \le \frac{288s^2\lambda^2}{n^2}
				+2\sigma^2\|\bL_n(\bXi)\|_2^2.								
	\end{align}
To complete the proof, we use the fact that $n\|\bL_n(\bXi)\|_2^2$ is a $\chi^2(p)$
random variable, which implies that with probability at least $1-\d/2$ it is bounded
from above by $2p+4\log(2/\delta)$. 
\end{proof}

\begin{proof}[Proof of \Cref{theoreme_bornesuperieure2}]
To ease notation, we set $\sigma_{n,p}^2 = (n-1)\sigma^2p/n$, 
$w_i=\big( 1-\sigma\gamma/\sqrt{(\|\bZ_i \|_2^2-\sigma_{n,p}^2)_+} \big)_+$,
$\tilde\bDelta_n = \tilde\bL_n-\bL_n(\bTheta)$ and $\bar\bxi_i= \bxi_i-\bar\bxi$.
Then,
	\begin{equation}
		\bZ_i = \bt_i + \tilde\bDelta_n + \sigma\bar\bxi_i.
	\end{equation}	
We first show that with high probability the weights $w_i$
vanish outside the support $S$.

\begin{lemma}\label{lemZ1}
In the event $\Omega_1\cap \Omega_2$, where
\begin{align}
\Omega_2 &= \Big\{\max_i\big|\|\bar\bxi_i\|_2^2-(1-\nicefrac1n)p\big|\le 0.5\gamma^2\Big\},
\end{align}
we have $w_i=0$ for every $i\not\in S$. Furthermore, under the condition
$\gamma^2 > 4\log(4n/\delta)+4 \{p\log(4n/\delta)\}^{1/2}$, the probability of
$\O_2$ is at least $1-\d/2$.
\end{lemma}

	Using equation~\eqref{Z0} and the fact that $w_i=0$ for every
	$i\not\in S$ in the event $\Omega_1\cap\Omega_2$, we get	
	\begin{align}
	\bL(\hat{\bTheta}-\bTheta)
		&= \sum_{i=1}^n \bZ_i w_i-\sum_{i\in S} \bt_i
		= \sum_{i\in S} \big(\bZ_i w_i -\bt_i\big).
	\end{align}
	Replacing $\bZ_i$ by $\bt_i + \tilde\bDelta_n + \sigma\bar\bxi_i$ and using the
	triangle inequality, we obtain
	\begin{align}
	\|\bL(\hat{\bTheta}-\bTheta) \|_2
		&\le \Big\|\sum_{i\in S} \big(\bZ_i-\sigma\bar\bxi_i\big) (w_i-1)\Big\|_2 +
			\Big\|\sum_{i\in S} \big(\bZ_i-\bt_i\big)+\sigma\bar\bxi_i(w_i-1)\Big\|_2\\
		&\le \sum_{i\in S} \|\bZ_i-\sigma\bar\bxi_i\|_2 |w_i-1| +
			\Big\|\sum_{i\in S} \big(\tilde\bDelta_n
			+ \sigma\bar\bxi_i w_i\big)\Big\|_2\\
		&\le \sum_{i\in S} \|\bZ_i-\sigma\bar\bxi_i\|_2 |w_i-1| +
			{s}\big\|\tilde\bDelta_n\big\|_2
			+ \sigma\Big\|\sum_{i\in S}\bar\bxi_i w_i\Big\|_2.\label{Z1}
	\end{align}	
	We will now evaluate the first and the third terms of the right-hand side.	
	\begin{lemma}\label{lemZ2}
	There is a sequence of standard Gaussian random variables 
	$\eta_1,\ldots,\eta_n$ such that in the event $\O_1\cap\O_2$, it holds
	\begin{align}
	\|\bZ_i - \sigma\bar\bxi_i\|_2|w_i-1|
		\le 2\s|\eta_i| + \big\{|2\sigma\bar\bxi_i^\top\tilde\bDelta_n|\big\}^{1/2} + 
			1.96\sigma\gamma.
		\label{Z2}
	\end{align}
	\end{lemma}

	Let us introduce the $p\times s$ matrix $\bXi_S = [\,\bar\bxi_i; i\in S\,]$ (the matrix
	obtained by concatenating the vectors $\bar\bxi_i$ with subscript $i$ running over $S$).
	Using the H\"older inequality, one can check that
	\begin{align}
	\sum_{i\in S}\big\{2\s|\bar\bxi_i^\top\tilde\bDelta_n|\big\}^{1/2}
			&\le \sqrt{2\sigma}\, s^{3/4} \bigg(\sum_{i\in S} |\bar\bxi_i^\top\tilde\bDelta_n|^2
				\bigg)^{1/4}\\
			&= \sqrt{2\sigma}\, s^{3/4} \|\bXi_S^\top\tilde\bDelta_n\|_2^{1/2}\\
			&\le \sqrt{2\sigma}\, s^{3/4} \|\bXi_S\|^{1/2}\|\tilde\bDelta_n\|_2^{1/2}.
			\label{Z3}
	\end{align}
	Finally, the third term in the right-hand side of \eqref{Z1} can be bounded as follows:
	\begin{equation}
		\Big\| \sum_{i\in S}\bar\bxi_iw_i\Big\|_2 = \|\bXi_S \bsw\|_2\le \sqrt{s}\,\|\bXi_S\|.
		\label{Z4}
	\end{equation}
	Combining \eqref{Z1}, \eqref{Z2}, \eqref{Z3} and \eqref{Z4}, we arrive at
	\begin{align}
	\|\bL(\hat{\bTheta}-\bTheta) \|_2
		&\le 2\sigma\sum_{i\in S}\big(|\eta_i| + \big\{2\s|\tilde\bDelta_n^\top 
		\bar\bxi_i|\big\}^{1/2}\big) +1.96\,\sigma s\gamma
			+ {s}\big\|\tilde\bDelta_n\big\|_2  + \sigma\sqrt{s}\,\|\bXi_S\|\\
		&\le 2\sigma\sum_{i\in S}|\eta_i|+2.21\,\sigma s\g+\sqrt{2\sigma}\, s^{3/4} \|\bXi_S\|^{1/2}
		\big\|\tilde\bDelta_n\big\|_2^{1/2}+\sigma\sqrt{s}\,\|\bXi_S\|\\
		&\le 2\sigma\sum_{i\in S}|\eta_i| +2.34\,\sigma s\gamma + 
		2\sigma\sqrt{s}\,\|\bXi_S\|.\label{Z5}
	\end{align}			
	According to \citep[Corollary~5.35]{vershynin_2012} (recalled
	in~\Cref{lemma_Wishart1} below for the reader's convenience) the event
	$\O_3 = \big\{ \|\bXi_S\| \le \sqrt{s}+\sqrt{p}+\{2\log(8/\delta)\}^{1/2} \big\}$
  has probability at least $1-\delta/8$. One can also check that with probability
	at least $1-\delta/8$, the event $\Omega_4 = \{\sum_{i\in S}|\eta_i| \le s + 
	\sqrt{s\log(8/\delta)}\}$ is realized. Assuming that $\delta\le 1/2$ so that 
	$\log (8/\delta)\le (4/3)\log (4/\delta)$, this implies that in $\O_1\cap\bar\O$ 
	with $\bar\O=\O_2\cap\O_3\cap\Omega_4$, we have
	\begin{align}
		\|\bL(\hat{\bTheta}-\bTheta) \|_2				
		&\le 4\sigma(\,s\gamma+ s + \sqrt{sp} +  \{2s\log(4/\delta)\}^{1/2}).
	\end{align}
	This completes the proof.
	\end{proof}

\section{Some technical lemmas}\label{section_proofs2}

\begin{lemma}\label{lemme_Pi}
	Let us introduce the projection matrix $\bfPi = \bfI_n - \frac1n \bfJ_n$, where $\bfI_n$ and $\bfJ_n$ are respectively the $n\times n$ identity matrix and
    the constant matrix with all the entries equal to $1$.
    Let $\bfU$ be a $p\times n$ matrix with columns $\bu_i\in\RR^p$ satisfying, for some set $S\subset \{1,\ldots,n\}$ and some real number $a>0$,
	\begin{equation}\label{cone_cond}
		\sum_{i\not\in S} \|\bu_i\|_2 \le a \sum_{i\in S} \|\bu_i\|_2,
	\end{equation}
	then
	\begin{equation}
		\|\bfU\bfPi\|_F^2 \ge \Big( 1 - \frac{(1+a)^2|S|}{n} \Big) \|\bfU\|_F^2.
	\end{equation}
\end{lemma}

\begin{proof}
	We denote by $\bu_i\in\RR^n$ the column vector corresponding to the $i$-th row of $\bfU$.
    On the one hand, since $\nicefrac1n\bfJ_n$ is an orthogonal projection matrix, we have by the Pythagorean theorem
	\begin{align}
		\|\bfPi\bu^i\|_2^2 &= \|\bu^i\|_2^2 -\frac1{n^2} \|\bfJ_n\bu^i\|_2^2.
	\end{align}
    In particular, this implies that
	\begin{align}\label{ineq:1}
		\|\bfU\bfPi\|_F^2 &= \|\bfU\|_F^2 -\frac1{n^2} \|\bfU\bfJ_n\|_F^2.
	\end{align}
	On the other hand,
	\begin{align}
		\|\bfU\bfJ_n\|_F &= \|\mathbf \bfU 1_n\mathbf 1_n^\top\|_F =\sqrt{n}\, \|\bfU \mathbf 1_n^\top\|_2 = \sqrt{n}\, \Big\|\sum_{i=1}^n \bu_i\Big\|_2
        \le \sqrt{n}\, \sum_{i=1}^n \|\bu_i\|_2.
	\end{align}
    Using \eqref{cone_cond} and the Cauchy-Schwarz inequality, we get
	\begin{align}
		\sum_{i=1}^n \|\bu_i\|_2 &\le (1+a)\sum_{i\in S} \|\bu_i\|_2 \le (1+a)|S|^{1/2}\Big(\sum_{i\in S} \|\bu_i\|^2_2\Big)^{1/2} \le (1+a)|S|^{1/2}\|\bfU\|_F.
	\end{align}
    This readily yields $\|\bfU\bfJ_n\|_F \le (1+a) (n|S|)^{1/2}\|\bfU\|_F$. Combining this inequality with \eqref{ineq:1}, we get the claim of the lemma.
\end{proof}

	\begin{proof}[Proof of \Cref{lemma:recursion}]
		We prove by inductive reasoning that for every $k$ we have
		$\epsilon_k^2\le p$ and \eqref{recursion:claim}. This is trivially
		true for $k=0$. Assume that these claims hold true for some given 
		value $k$. Let us check them for the value $k+1$. 
		From recursion \eqref{eq:recursion}, one can infer that  
		\begin{align}\label{eq:recursion1a}
				\epsilon_{k+1} \le \frac{16.5s}{n}(\epsilon_k\vee p^{1/4}\epsilon_k^{1/2}) +a
				\le \frac{33s}{n}(\epsilon_k\vee p^{1/4}\epsilon_k^{1/2})\vee 2a.
		\end{align}
		The conditions of the lemma directly imply that $\epsilon_{k+1}^2\le p$. Therefore,
		\begin{align}\label{eq:recursion1b}
				\epsilon_{k+1} \le \frac{33sp^{1/4}\epsilon_k^{1/2}}{n}\vee 2a.
		\end{align}
		Applying inequality \eqref{recursion:claim}, we get the claim.
		
		One can note that even in the case $\epsilon_k>\sqrt{p}$, we get 
		from \eqref{eq:recursion1a} that
		\begin{align}\label{eq:recursion2}
				\epsilon_{k+1} \le \frac{1}{2}\epsilon_k +a.
		\end{align}
		Therefore, if the preliminary estimator is not good enough to guarantee that 
		$\epsilon_0\le \sqrt{p}$, after a number of steps at most logarithmic in 
		$\epsilon_0/(\sqrt{p}-2a)$, we will get an error $\epsilon_k$ smaller than
		$\sqrt{p}$.
	\end{proof}

	\begin{proof}[Proof of \Cref{lemZ1}]
		It follows from the definition of $w_i$ that, for every $i\not\in S$, $w_i>0$ 
		is equivalent to
		$$
				\|\sigma\bar\bxi_i+\tilde\bDelta_n\|_2^2 > \sigma^2(p+\gamma^2).
		$$
		In view of the triangle inequality, this implies that
		$$
				\sigma\|\bar\bxi_i\|_2 + \|\tilde\bDelta_n\|_2 > \sigma\sqrt{p+\gamma^2}.
		$$			
		In the event $\Omega_1\cap\Omega_2$, the last inequality implies
		$$		
				\sqrt{p+0.5\gamma^2} + \frac{\gamma^2}{4\sqrt{p+\gamma^2}}> \sqrt{p+\gamma^2}.
		$$
		It is easy to see that the last inequality is never true, implying thus that
		$w_i=0$. Indeed,
		\begin{align}
				\sqrt{p+\gamma^2} -\sqrt{p+0.5\gamma^2}  
					& = \frac{0.5\gamma^2}{\sqrt{p+\gamma^2} +\sqrt{p+0.5\gamma^2}}\\
					& > \frac{0.5\gamma^2}{\sqrt{p+\gamma^2} +\sqrt{p+\gamma^2}}
					= \frac{\gamma^2}{4\sqrt{p+\gamma^2}}.
		\end{align}
		The fact that the probability of $\O_2$ is at least $1-\d/2$ is
		a consequence of the tail bound of a chi-squared random variable and the
		union bound.
	\end{proof}

	\begin{proof}[Proof of \Cref{lemZ2}] The definition of $\bZ_i$ yields
	\begin{align}
	\|\bZ_i-\sigma\bar\bxi_i\|_2|w_i-1|  &= \|\bZ_i-\sigma\bar\bxi_i\|_2 \min\bigg\{ 1,
	\frac{\sigma\gamma}{\big(\|\bZ_i\|^2_2-\sigma_{n,p}^2\big)^{1/2}_+} \bigg\}.
	\end{align}
	Let us introduce the random variables $\eta_i = 
	\frac{n}{n-1}\bt_i^\top\bar\bxi_i/\|\bt_i\|_2$. It is clear that $\eta_i$ is Gaussian
	with zero mean and unit variance. Repeated use of the triangle inequality leads to
	\begin{align}
	\|\bZ_i\|_2^2
		&\ge \|\bZ_i-\sigma\bar\bxi_i\|_2^2  + \|\sigma\bar\bxi_i\|_2^2 + 2\sigma\bar\bxi_i^\top(\bt_i
		+\tilde\bDelta_n) \\
		&\ge \|\bZ_i-\sigma\bar\bxi_i\|_2^2  + \|\sigma\bar\bxi_i\|_2^2 -2\sigma\|\bt_i\|_2|\eta_i|
		+2\sigma\bar\bxi_i^\top\tilde\bDelta_n\\
		&\ge \|\bZ_i-\sigma\bar\bxi_i\|_2^2  + \|\sigma\bar\bxi_i\|_2^2 -2\sigma\|\bZ_i-\sigma\bar\bxi_i\|_2
		|\eta_i| -2\sigma\|\tilde\bDelta_n\|_2|\eta_i|+2\sigma\bar\bxi_i^\top\tilde\bDelta_n\\
		&\ge (\|\bZ_i-\sigma\bar\bxi_i\|_2-\s|\eta_i|)^2 -\sigma^2\eta_i^2 + \|\sigma\bar\bxi_i\|_2^2
		-2\sigma\|\tilde\bDelta_n\|_2|\eta_i|+2\sigma\bar\bxi_i^\top\tilde\bDelta_n.
	\end{align}
	One can check the following simple fact: if $a,b,c,d>0$ then
	$$
	a\min\Big(1,\frac{d}{\{(a-b)^2-c\}_+^{1/2}}\Big)\le  b + d+\sqrt{c_+}.
	$$
	Taking in this inequality $a = \|\bZ_i-\sigma\bar\bxi_i\|_2$, 
	$c = \sigma^2\eta_i^2 - \|\sigma\bar\bxi_i\|_2^2
		+2\sigma\|\tilde\bDelta_n\|_2|\eta_i|-2\sigma\bar\bxi_i^\top\tilde\bDelta_n + 
		\sigma^2_{n,p}$, $b=\s|\eta_i|$ and $d=\sigma\g$, we arrive at
	\begin{align}
	\|\bZ_i-\sigma\bar\bxi_i\|_2|w_i-1|
		&\le 2\s|\eta_i| + \sigma\g+ \|\tilde\bDelta_n\|_2\\
		&\qquad+\sqrt{(\sigma_{n,p}^2-\sigma^2\|\bar\bxi_i\|_2^2)_+}
		+\sqrt{(2\sigma\bar\bxi_i^\top\tilde\bDelta_n -\|\tilde\bDelta_n\|_2^2)_+}.
	\end{align}
	In the event $\O_1\cap\O_2$, we upper bound $\|\tilde\bDelta_n\|_2$ and
	$|\sigma_{n,p}^2-\sigma^2\|\bar\bxi_i\|_2^2|$ respectively by $0.25\sigma\gamma$ 
	and $0.5\sigma^2\gamma^2$. This leads to the claim of the lemma.
\end{proof}

\section{Tail bounds}\label{section_tails}

In this section, we recall  well-known results on the tails of
some random variables appearing in the analysis of the Gaussian models
of the previous sections.

\begin{lemma}\label{lemme_chideux}
	If $\eta$ is a random variable drawn from the $\chi^2_d$ distribution, then for every $x>0$
	\begin{equation}
		\prob\big(\eta \ge d+ x\big) \le e^{-\frac1{16d} x (x\wedge 4d)} .
	\end{equation}
\end{lemma}

\begin{lemma}\label{lemme_esperancechideux}
	If $\eta$ is a random variable drawn from the $\chi^2_d$ distribution, then for every $x>0$ and $d\ge 2$,
	\begin{equation}
	\esp\big(\eta \fcar_{\eta\ge d+x}\big) \le 2d\,e^{-x^2/32d} \,\fcar_{x< 4d} +
	2xe^{-x/4} \fcar_{x\ge 4d}.
	\end{equation}	
\end{lemma}

\begin{proof}
	First assume that $x\ge 4d$. Combining the relation
	\begin{align}
		\esp\big(\eta\, \fcar_{\eta\ge d+x}\big) &= (d+x)\prob(\eta\ge d+x)+\int_x^\pinf \prob(\eta\ge d+t)\,dt
	\end{align}
	with Lemma~\ref{lemme_chideux}, we get
	\begin{align}
		\esp\big(\eta\, \fcar_{\eta\ge d+x}\big)
		&\le (d+x) e^{-x/4} + \int_x^\pinf e^{-t/4}\, dt\\
		&\le (4+d+x)e^{-x/4}\le (4d+x)e^{-x/4}\le 2xe^{-x/4}. \phantom{\int_0^\pinf}
	\end{align}
	Now assume that $x<4d$. Then, using the Cauchy-Schwarz inequality and
	Lemma~\ref{lemme_chideux},
	\begin{align}
	\esp\big(\eta\, \fcar_{\eta\ge d+x}\big) \le \sqrt{d(d+2)} e^{-x^2/32d}.
	\end{align}
\end{proof}

\begin{lemma}\label{lemme_probabiliteA}
	Denote $\bfPi=I_n-\frac1n \bfJ_n$  where $\bfI_n$ is the identity matrix in dimension $n$ and $\bfJ_n$ is the constant matrix with only $1$ coefficients, and assume that
	\begin{equation}
		\bxi_1,\ldots,\bxi_n \simiid \mathcal{N}(0,\bfI_p).
	\end{equation}
	Then, with probability at least $1-\d$, the matrix $\bXi = [\bxi_1,\ldots,\bxi_n]$ 
	satisfies
	\begin{equation}
		\max_{i=1,\ldots,n}  \big\|(\bXi\bfPi)_i\big\|_2^2 \le 2p+16\log(n/\delta).
	\end{equation}
\end{lemma}

\begin{proof}
	We first notice that
	\begin{equation}
		(\bXi\bfPi)_i = \bxi_i - \frac1n\sum_{j=1}^n \bxi_j\ \simiid\ 
		\mathcal{N}\big(0,{\textstyle\frac{n-1}{n}}\bfI_p\big).
	\end{equation}
	This implies that the random variable
	$\big\|(\bXi\bfPi)_i\big\|_2^2$ is drawn from the $\frac{n-1}{n} \chi^2_p$ 
	distribution, and the result follows from~\Cref{lemme_chideux}.
\end{proof}

\begin{lemma}[Corollary~5.35 in~\citet*{vershynin_2012}]\label{lemma_Wishart1}
	Assume that $\bfA$ is a $N\times n$ random matrix with independent standard Gaussian entries.
	Then, for any $t\ge 0$,  with probability at least $1-2e^{-t^2/2}$, it holds that
	\begin{equation}
	\|\bfA\| \le \sqrt{N}+\sqrt{n}+t.
	\end{equation}
\end{lemma}

We deduce from this the following lemma.

\begin{lemma}\label{lemma_Wishart2}
	If $\bfA$ is a $N\times n$ random matrix with independent standard Gaussian
	entries, then
	$\esp[\|\bfA\|^2] \le 3N+3n+12$.	
\end{lemma}

\begin{proof}
	It is clear that
	\begin{align}
	\esp[\|\bfA\|^2] &\le 3(N+n) + \int_0^\pinf \prob(\|\bfA\|^2>3(N+n)+x)\, dx \\
	&\le 3(N+n) + 12\,\int_0^\pinf te^{-t^2/2}\, dt.
	\end{align}
	The result follows from the fact that the last integral is equal to one.
\end{proof}

\section*{Acknowledgments}
O.\ Collier's research has been conducted as part of the project Labex MME-DII (ANR11-LBX-0023-01).
The work of A. Dalalyan was partially supported by the grant Investissements d'Avenir
(ANR-11IDEX-0003/Labex Ecodec/ANR-11-LABX-0047).

\bibliography{refs}

\begin{thebibliography}{46}
\providecommand{\natexlab}[1]{#1}
\providecommand{\url}[1]{\texttt{#1}}
\expandafter\ifx\csname urlstyle\endcsname\relax
  \providecommand{\doi}[1]{doi: #1}\else
  \providecommand{\doi}{doi: \begingroup \urlstyle{rm}\Url}\fi

\bibitem[Balmand and Dalalyan(2015)]{BalDal15b}
Samuel Balmand and Arnak~S. Dalalyan.
\newblock Convex programming approach to robust estimation of a multivariate
  gaussian model.
\newblock submitted 1512.04734, arXiv, December 2015.

\bibitem[Bickel and Ritov(1988)]{BickelRitov}
P.~J. Bickel and Y.~Ritov.
\newblock Estimating integrated squared density derivatives: Sharp best order
  of convergence estimates.
\newblock \emph{Sankhy?: The Indian Journal of Statistics, Series A
  (1961-2002)}, 50\penalty0 (3):\penalty0 381--393, 1988.

\bibitem[Bunea et~al.(2014)Bunea, Lederer, and She]{Bunea2014}
F.~Bunea, J.~Lederer, and Y.~She.
\newblock The group square-root lasso: Theoretical properties and fast
  algorithms.
\newblock \emph{IEEE Transactions on Information Theory}, 60\penalty0
  (2):\penalty0 1313--1325, Feb 2014.

\bibitem[Butucea and Comte(2009)]{butucea2009}
C.~Butucea and F.~Comte.
\newblock Adaptive estimation of linear functionals in the convolution model
  and applications.
\newblock \emph{Bernoulli}, 15\penalty0 (1):\penalty0 69--98, 02 2009.
\newblock \doi{10.3150/08-BEJ146}.

\bibitem[Cai and Low(2004)]{cai2004}
T.~Tony Cai and Mark~G. Low.
\newblock Minimax estimation of linear functionals over nonconvex parameter
  spaces.
\newblock \emph{Ann. Statist.}, 32\penalty0 (2):\penalty0 552--576, 04 2004.

\bibitem[Cai and Low(2005)]{cai2005}
T.~Tony Cai and Mark~G. Low.
\newblock On adaptive estimation of linear functionals.
\newblock \emph{Ann. Statist.}, 33\penalty0 (5):\penalty0 2311--2343, 10 2005.

\bibitem[Cai and Low(2006)]{cai2006}
T.~Tony Cai and Mark~G. Low.
\newblock Optimal adaptive estimation of a quadratic functional.
\newblock \emph{Ann. Statist.}, 34\penalty0 (5):\penalty0 2298--2325, 10 2006.

\bibitem[Cai and Low(2011)]{cai2011}
T.~Tony Cai and Mark~G. Low.
\newblock Testing composite hypotheses, hermite polynomials and optimal
  estimation of a nonsmooth functional.
\newblock \emph{Ann. Statist.}, 39\penalty0 (2):\penalty0 1012--1041, 04 2011.

\bibitem[{Chen} et~al.(2015){Chen}, {Gao}, and {Ren}]{Gao2015}
M.~{Chen}, C.~{Gao}, and Z.~{Ren}.
\newblock {Robust Covariance and Scatter Matrix Estimation under Huber's
  Contamination Model}.
\newblock \emph{ArXiv e-prints, to appear in the Annals of Statistics}, 2015.

\bibitem[Chen et~al.(2016)Chen, Gao, and Ren]{chen2016}
Mengjie Chen, Chao Gao, and Zhao Ren.
\newblock A general decision theory for huber's $\epsilon$-contamination model.
\newblock \emph{Electron. J. Statist.}, 10\penalty0 (2):\penalty0 3752--3774,
  2016.

\bibitem[Cherapanamjeri et~al.(2016)Cherapanamjeri, Gupta, and
  Jain]{CherapanamjeriG16}
Yeshwanth Cherapanamjeri, Kartik Gupta, and Prateek Jain.
\newblock Nearly-optimal robust matrix completion.
\newblock \emph{CoRR}, abs/1606.07315, 2016.

\bibitem[Chesneau and Hebiri(2008)]{Hebiri_2008}
Christophe Chesneau and Mohamed Hebiri.
\newblock Some theoretical results on the grouped variables {L}asso.
\newblock \emph{Math. Methods Statist.}, 17\penalty0 (4):\penalty0 317--326,
  2008.

\bibitem[{Collier} et~al.(2016){Collier}, {Comminges}, {Tsybakov}, and
  {Verz{\'e}len}]{ColComTV}
O.~{Collier}, L.~{Comminges}, A.~B. {Tsybakov}, and N.~{Verz{\'e}len}.
\newblock {Optimal adaptive estimation of linear functionals under sparsity}.
\newblock \emph{ArXiv e-prints, ArXiv:1611.09744}, November 2016.

\bibitem[Collier and Dalalyan(2015)]{CD11c}
Olivier Collier and Arnak~S. Dalalyan.
\newblock Curve registration by nonparametric goodness-of-fit testing.
\newblock \emph{J. Statist. Plann. Inference}, 162:\penalty0 20--42, July 2015.

\bibitem[Collier and Dalalyan(2018)]{CollierDal3}
Olivier Collier and Arnak~S. Dalalyan.
\newblock Estimating linear functionals of a sparse family of poisson means.
\newblock \emph{Statistical Inference for Stochastic Processes}, Feb 2018.

\bibitem[Collier et~al.(2017)Collier, Comminges, and Tsybakov]{collier2017}
Olivier Collier, La{\"e}titia Comminges, and Alexandre~B. Tsybakov.
\newblock Minimax estimation of linear and quadratic functionals on sparsity
  classes.
\newblock \emph{Ann. Statist.}, 45\penalty0 (3):\penalty0 923--958, 2017.

\bibitem[Comminges and Dalalyan(2012)]{comminges2012}
La{\"e}titia Comminges and Arnak~S. Dalalyan.
\newblock Tight conditions for consistency of variable selection in the context
  of high dimensionality.
\newblock \emph{Ann. Statist.}, 40\penalty0 (5):\penalty0 2667--2696, 10 2012.

\bibitem[Comminges and Dalalyan(2013)]{comminges2013}
La{\"e}titia Comminges and Arnak~S. Dalalyan.
\newblock Minimax testing of a composite null hypothesis defined via a
  quadratic functional in the model of regression.
\newblock \emph{Electron. J. Statist.}, 7:\penalty0 146--190, 2013.

\bibitem[Dalalyan and Chen(2012)]{DalalyanC12}
Arnak~S. Dalalyan and Yin Chen.
\newblock Fused sparsity and robust estimation for linear models with unknown
  variance.
\newblock In \emph{Advances in Neural Information Processing Systems 25: NIPS},
  pages 1268--1276, 2012.

\bibitem[Dalalyan and Keriven(2012)]{DK12}
Arnak~S. Dalalyan and Renaud Keriven.
\newblock Robust estimation for an inverse problem arising in multiview
  geometry.
\newblock \emph{J. Math. Imaging Vision}, 43\penalty0 (1):\penalty0 10--23,
  2012.

\bibitem[Devroye et~al.(2016)Devroye, Lerasle, Lugosi, and
  Oliveira]{devroye2016}
Luc Devroye, Matthieu Lerasle, Gabor Lugosi, and Roberto~I. Oliveira.
\newblock Sub-gaussian mean estimators.
\newblock \emph{Ann. Statist.}, 44\penalty0 (6):\penalty0 2695--2725, 12 2016.
\newblock \doi{10.1214/16-AOS1440}.

\bibitem[Donoho and Montanari(2016)]{Donoho2016}
David Donoho and Andrea Montanari.
\newblock High dimensional robust m-estimation: asymptotic variance via
  approximate message passing.
\newblock \emph{Probability Theory and Related Fields}, 166\penalty0
  (3):\penalty0 935--969, Dec 2016.

\bibitem[Donoho and Nussbaum(1990)]{DONOHO1990290}
David~L Donoho and Michael Nussbaum.
\newblock Minimax quadratic estimation of a quadratic functional.
\newblock \emph{Journal of Complexity}, 6\penalty0 (3):\penalty0 290 -- 323,
  1990.

\bibitem[Donoho and Liu(1987)]{donoho1987minimax}
D.L. Donoho and R.C. Liu.
\newblock \emph{On Minimax Estimation of Linear Functionals}.
\newblock Technical report (University of California, Berkeley. Department of
  Statistics). Department of Statistics, University of California, 1987.

\bibitem[Efromovich and Low(1994)]{Efromovich1994}
Sam Efromovich and Mark~G. Low.
\newblock Adaptive estimates of linear functionals.
\newblock \emph{Probability Theory and Related Fields}, 98\penalty0
  (2):\penalty0 261--275, Jun 1994.

\bibitem[Goldenshluger and Nemirovski(1997)]{Golden}
A.~Goldenshluger and A.~Nemirovski.
\newblock On spatial adaptive estimation of nonparametric regression.
\newblock \emph{Math. Meth. Statistics}, 6:\penalty0 135--170, 1997.

\bibitem[Golubev and Levit(2004)]{Gol04}
Y~Golubev and B~Levit.
\newblock An oracle approach to adaptive estimation of linear functionals in a
  gaussian model.
\newblock \emph{Mathematical Methods of Statistics}, 13\penalty0 (01):\penalty0
  392--408, 2004.

\bibitem[Huber(1964)]{huber1964}
Peter~J. Huber.
\newblock Robust estimation of a location parameter.
\newblock \emph{Ann. Math. Statist.}, 35\penalty0 (1):\penalty0 73--101, 1964.

\bibitem[Juditsky and Nemirovski(2009)]{juditsky2009}
Anatoli~B. Juditsky and Arkadi~S. Nemirovski.
\newblock Nonparametric estimation by convex programming.
\newblock \emph{Ann. Statist.}, 37\penalty0 (5A):\penalty0 2278--2300, 10 2009.
\newblock \doi{10.1214/08-AOS654}.
\newblock URL \url{https://doi.org/10.1214/08-AOS654}.

\bibitem[Klemel{a} and Tsybakov(2001)]{klemela2001}
Jussi Klemel{a} and Alexandre~B. Tsybakov.
\newblock Sharp adaptive estimation of linear functionals.
\newblock \emph{Ann. Statist.}, 29\penalty0 (6):\penalty0 1567--1600, 12 2001.
\newblock URL \url{https://doi.org/10.1214/aos/1015345955}.

\bibitem[Klopp et~al.(2017)Klopp, Lounici, and Tsybakov]{Klopp2017}
Olga Klopp, Karim Lounici, and Alexandre~B. Tsybakov.
\newblock Robust matrix completion.
\newblock \emph{Probability Theory and Related Fields}, 169\penalty0
  (1):\penalty0 523--564, Oct 2017.

\bibitem[Koshevnik and Levit(1977)]{Koshevnik}
Yu.~A. Koshevnik and B.~Ya. Levit.
\newblock On a non-parametric analogue of the information matrix.
\newblock \emph{Theory of Probability \& Its Applications}, 21\penalty0
  (4):\penalty0 738--753, 1977.
\newblock \doi{10.1137/1121087}.

\bibitem[Laurent and Massart(2000)]{laurent2000}
B.~Laurent and P.~Massart.
\newblock Adaptive estimation of a quadratic functional by model selection.
\newblock \emph{Ann. Statist.}, 28\penalty0 (5):\penalty0 1302--1338, 10 2000.
\newblock \doi{10.1214/aos/1015957395}.
\newblock URL \url{https://doi.org/10.1214/aos/1015957395}.

\bibitem[Laurent et~al.(2008)Laurent, Ludena, and Prieur]{laurent2008}
B{\'e}atrice Laurent, Carenne Ludena, and Cl{\'e}mentine Prieur.
\newblock Adaptive estimation of linear functionals by model selection.
\newblock \emph{Electron. J. Statist.}, 2:\penalty0 993--1020, 2008.
\newblock \doi{10.1214/07-EJS127}.

\bibitem[Lecu{\'e} and Lerasle(2017)]{LL17}
{Guillaume} Lecu{\'e} and {Matthieu} Lerasle.
\newblock {Learning from MOM's principles: Le Cam's approach}.
\newblock \emph{ArXiv e-prints}, January 2017.

\bibitem[Lepski et~al.(1999)Lepski, Nemirovski, and Spokoiny]{Lepski1999}
O.~Lepski, A.~Nemirovski, and V.~Spokoiny.
\newblock On estimation of the {$L_r$} norm of a regression function.
\newblock \emph{Probability Theory and Related Fields}, 113\penalty0
  (2):\penalty0 221--253, Feb 1999.

\bibitem[Lepskii(1991)]{Lepski91}
O.~V. Lepskii.
\newblock On a problem of adaptive estimation in gaussian white noise.
\newblock \emph{Theory of Probability \& Its Applications}, 35\penalty0
  (3):\penalty0 454--466, 1991.

\bibitem[{Lerasle} and {Oliveira}(2011)]{Oliveira}
M.~{Lerasle} and R.~I. {Oliveira}.
\newblock {Robust empirical mean Estimators}.
\newblock \emph{ArXiv e-prints}, December 2011.

\bibitem[Lin and Zhang(2006)]{Lin_2006}
Yi~Lin and Hao~Helen Zhang.
\newblock Component selection and smoothing in multivariate nonparametric
  regression.
\newblock \emph{Ann. Statist.}, 34\penalty0 (5):\penalty0 2272--2297, 2006.

\bibitem[Lounici et~al.(2011)Lounici, Pontil, van~de Geer, and
  Tsybakov]{lounici2011}
Karim Lounici, Massimiliano Pontil, Sara van~de Geer, and Alexandre~B.
  Tsybakov.
\newblock Oracle inequalities and optimal inference under group sparsity.
\newblock \emph{Ann. Statist.}, 39\penalty0 (4):\penalty0 2164--2204, 2011.

\bibitem[Meier et~al.(2009)Meier, van~de Geer, and B{\"u}hlmann]{Meier_2009}
Lukas Meier, Sara van~de Geer, and Peter B{\"u}hlmann.
\newblock High-dimensional additive modeling.
\newblock \emph{Ann. Statist.}, 37\penalty0 (6B):\penalty0 3779--3821, 2009.

\bibitem[Minsker(2015)]{minsker2015}
Stanislav Minsker.
\newblock Geometric median and robust estimation in banach spaces.
\newblock \emph{Bernoulli}, 21\penalty0 (4):\penalty0 2308--2335, 11 2015.
\newblock \doi{10.3150/14-BEJ645}.

\bibitem[Nguyen and Tran(2013)]{Nguyen}
N.~H. Nguyen and T.~D. Tran.
\newblock Robust lasso with missing and grossly corrupted observations.
\newblock \emph{IEEE Transactions on Information Theory}, 59\penalty0
  (4):\penalty0 2036--2058, 2013.

\bibitem[Vershynin(2012)]{vershynin_2012}
Roman Vershynin.
\newblock \emph{Introduction to the non-asymptotic analysis of random
  matrices}, pages 210--268.
\newblock Cambridge University Press, 2012.
\newblock \doi{10.1017/CBO9780511794308.006}.

\bibitem[{Verzelen} and {Gassiat}(2016)]{Gassiat}
N.~{Verzelen} and E.~{Gassiat}.
\newblock {Adaptive estimation of High-Dimensional Signal-to-Noise Ratios}.
\newblock \emph{ArXiv e-prints}, February 2016.

\bibitem[Yuan and Lin(2006)]{Yuan_Lin_2006}
Ming Yuan and Yi~Lin.
\newblock Model selection and estimation in regression with grouped variables.
\newblock \emph{J. R. Stat. Soc. Ser. B Stat. Methodol.}, 68\penalty0
  (1):\penalty0 49--67, 2006.

\end{thebibliography}

\end{document}